\documentclass{article}[a4paper,8pt,twosided]
\usepackage[utf8]{inputenc}
\usepackage[T1]{fontenc}
\usepackage{times}
\usepackage{fullpage}
\usepackage{amsmath,amsfonts,amssymb,amsthm,psfrag,wrapfig,float}
\usepackage{authblk}
\usepackage{graphicx}
\usepackage[font=small,labelfont=bf]{caption}
\usepackage{bbm}
\usepackage{color}
\usepackage[colorlinks=true,
			linkcolor={blue},
			citecolor={blue},
			urlcolor={blue}]{hyperref}
\usepackage[table]{xcolor}
\usepackage{euscript,graphicx}
\usepackage{float}
\usepackage{subfigure}
\usepackage{xspace}
\usepackage{booktabs}

\newcommand{\bRb}{\mathbb{R}}

\newtheorem{thm}{Theorem}[section]
\newtheorem{prop}[thm]{Proposition}

\newtheorem{lem}[thm]{Lemma}
\newtheorem{rem}[thm]{Remark}

\newtheorem{defn}[thm]{Definition}

\newcommand{\cGc}{\mathcal{G}}

\newcommand{\gaga}{\left|\left|}
\newcommand{\drdr}{\right|\right|}

\newcommand{\bal}{\begin{align}}
\newcommand{\eal}{\end{align}}
\newcommand{\beq}{\begin{equation}}
\newcommand{\eeq}{\end{equation}}

\newcommand{\ba}{\begin{align*}}
\newcommand{\be}{\begin{equation*}}
\newcommand{\ee}{\end{equation*}}

\newcommand{\EE}{\mathbb{E}}
\newcommand{\PP}{\mathbb{P}}

\newcommand{\h}{{h}} 

 \title{A Kramers' type law for the first collision-time of two self-stabilizing diffusions and of their particle approximations}
\author[1]{Jean-Fran\c{c}ois Jabir}
\author[2]{Julian Tugaut}
\affil[1]{HSE University, Department of Statistics and Data Analysis $\&$ Laboratory of Stochastic Analysis and its Applications, Moscow, Russia. jjabir@hse.ru}
\affil[2]{Universit\'e Jean Monnet, Saint-Etienne, $\&$ Institut Camille Jordan, Lyon,
France. julian.tugaut@univ-st-etienne.fr}

\begin{document}
\maketitle
\pagestyle{empty}
\setlength{\textwidth}{14cm}
\begin{abstract}
The present work investigates the asymptotic behaviors, at the zero-noise limit, of the first collision-time and first collision-location related to a pair of self-stabilizing diffusions and of their related particle approximations. These asymptotic are considered in a peculiar framework {\color{black}where} diffusions evolve in a double-wells landscape 
{\color{black}where} collisions manifest due to the combined action of the Brownian motions driving each diffusion and the action of a self-stabilizing kernel.
As the Brownian effects vanish, we show that first collision-times grow at an explicit exponential rate and that the related collision-locations persist at a special point in space.
These results are mainly obtained by linking collision phenomena for diffusion processes with exit-time problems of random perturbed dynamical systems, and by exploiting Freidlin-Wentzell's LDP approach to solve these exit-time problems. Importantly, we consider two distinctive situations: the one-dimensional case (where true collisions can be directly studied) and the general multidimensional case (where collisions are required to be enlarged).
\end{abstract}
\medskip

{\bf Key words:} {\color{black}Noise-induced collisions; Asymptotic of McKean-Vlasov diffusion at small-noise limit; Freidlin-Wentzell theory of Gaussian perturbed dynamical systems.}
\par\medskip

{\bf 2020 AMS subject classifications:} Primary: 60H10.  
Secondary: 60J60,  60K35,  37A50.  
\par\medskip

\section{Introduction}

\subsection{Setting}
In this paper, we are interested in estimating the zero-noise limit of the first collision-time and first collision-location (or an $\epsilon$-approximation of these quantities) of two nonlinear self-stabilizing diffusions, $X=(X_t)_{t\ge 0}$ and~$Y=(Y_t)_{t\ge 0}$, whose dynamics are given by:
\begin{subequations}\label{MV}
\begin{equation}
\label{MV1}
\left\{
\begin{aligned}
&X_t=x_1+\sigma B_t-\int_0^t\Big(\nabla V\left(X_s\right)+\int \nabla F\left(X_s-x\right)\,\mu^X(s,dx)\Big)ds\,,\\
&\mu^X(t)=\text{Law}(X_t)\,,\,t\ge 0\,,
\end{aligned}
\right.
\end{equation}
and
\begin{equation}
\label{MV2}
\left\{
\begin{aligned}
&Y_t=x_2+\sigma \widetilde{B}_t-\int_0^t\Big(\nabla V\left(Y_s\right)+\int \nabla F\left(Y_s-y\right)\,\mu^Y(s,dy)\Big)\,ds\,,\\
&\mu^Y(t)=\text{Law}(Y_t)\,,\,t\ge 0\,.
\end{aligned}
\right.
\end{equation}
\end{subequations}
Here and after, $x_1$ and $x_2$ feature two deterministic initial conditions,  $\sigma$ a positive constant and $B$ and $\widetilde{B}$ denote two independent $\mathbb R^d$-Brownian motions. The derive functions in \eqref{MV1} and~\eqref{MV2} are characterized by the potentials functions $V$ and $F$ which, in addition to be smooth, will be assumed to generate,  on the one hand, a bistable landscape and, on the second hand, a stabilization effect which settles down each dynamic in a given steady region (our exact setting is detailed in Assumptions ${\bf (A)}$ below).
 
Self-stabilizing diffusions define particular instances of McKean-Vlasov models with contractive nonlinear coefficients. The latter, historically introduced in \cite{McKean-66}, \cite{McKean-67}, broadly refer to a class of SDEs where coefficients depend on the distribution itself of the solution to the equation. McKean-Vlasov models arise with the probabilistic interpretation of nonlinear PDEs and as the mean-field - or large population - limit of interacting stochastic particle systems; we refer the interested reader to \cite{Bossy-03}, \cite{JabinWang-17}, \cite{Chaintron-Diez-21} for exhaustive surveys on these topics. For their parts, the dynamics \eqref{MV1} and \eqref{MV2} emerge as the natural large population limit ($N\uparrow \infty$) of the family of exchangeable interacting particle systems~$(X^{1,N},\cdots,X^{N,N})$ and $(Y^{1,N},\cdots,Y^{N,N})$, given by: 
 \begin{subequations}\label{particles}
\begin{equation}
\label{Ibis}
\left\{
\begin{aligned}
&X_t^{i,N}=x_1+\sigma B_t^i-\int_0^t\Big(\nabla V\left(X_s^{i,N}\right)+\frac{1}{N}\sum_{j=1}^N\nabla F\big(X_s^{i,N}-X_s^{j,N}\big)\Big)ds\,,\\
&\,t\ge 0\,,1\le i\le N\,,
\end{aligned}
\right.
\end{equation}
and
\begin{equation}
\label{IIbis}
\left\{
\begin{aligned}
&Y_t^{i,N}=x_2+\sigma\widetilde{B}_t^i-\int_0^t\Big(\nabla V(Y_s^{i,N})+\frac{1}{N}\sum_{j=1}^N\nabla F\big(Y_s^{i,N}-Y_s^{j,N})\Big)ds\,,\\
&t\ge 0\,,1\le i\le N\,,
\end{aligned}
\right.
\end{equation}
\end{subequations}
the driving noises $(B^1,\cdots,B^N)$ and $(\widetilde{B}^1,\cdots,\widetilde{B}^N)$ denoting here mutually independent copies of $B$ and $\widetilde{B}$.  
 
 The motions of the self-stabilizing diffusions \eqref{MV1} and \eqref{MV2} are governed by three mechanisms: the diffusive effect of the Brownian motions whose intensities are parameterized by~$\sigma$; the action of the external potential force $-\nabla V$; and the action of an internal potential force~$-\nabla F$, characterizing at the meso-scopic scale, the interactions driving \eqref{Ibis} and \eqref{IIbis}.  
 In the absence of an internal force, the diffusions correspond to stochastic gradient flows whose long-time behaviors, for a non-trivial potential {\color{black}$V$} with suitably growth, are governed by the Gibbs measure $(R_\sigma)^{-1}\exp\{- 2V/{\sigma^2} \}$~-~where~$R_\sigma$ is standing for a renormalizing constant~-~regardless of the convexity of $V$ and the initial states. On the other hand, in the absence of an external potential, the force field $-\nabla F$ may induce a long-time stabilization effect on the Brownian diffusion towards an invariant probability measure depending only on the first initial moment of the diffusion process (see~\cite{BRTV}, \cite{BRV}).  Combined, the potentials can create a discrepancy generating multiple invariant probability measures. A prototypical example where this situation occurs is given by the one-dimensional double-wells potential~$V(x):=x^4/4-x^2/2$ and the mean-attracting force generated by $F(x):=\alpha x^2/2$ and $\alpha>0$ (we refer the interested reader to the seminal papers~\cite{Kramers-40} and ~\cite{Dawson-83}, and the references therein, for the practical and theoretical interests of these potentials).
 The resulting model illustrates the situation where different stationary regimes emerge depending on the parameters $\alpha$ and $\sigma$. The wells $\lambda_1=-1$ and $\lambda_2=1$ and the ``bump'' $\lambda_0=0$ characterize the three possible attractive points for the dynamics which dominate long-time asymptotic as $\sigma\downarrow 0$. Whenever~$\sigma$ is larger than a certain threshold, convergence to the unique invariant measure is ensured while, whenever~$\sigma$ is below this threshold, three invariant probability measures emerge, two of them being concentrated around $\lambda_1$ and $\lambda_2$ (\cite[Sections 3 and~4]{Dawson-83}). In the case where $\alpha$ is itself large enough (a case that we will below refer to as \textit{synchronization}), long-time behavior can be analyzed at very small-noise intensity.   
  
  The question of the long-time behavior of self-stabilizing diffusions along establishing explicitly the asymptotic ``large population, large time'' of the related particle systems,  
  has been intensively investigated under prior assumptions ensuring uniqueness of the invariant measure, see e.g. \cite{BCCP}, \cite{Malrieu2001}, \cite{BGG2}, \cite{CMV2003}, \cite{CGM}, \cite{BoGuiMa-10}. In the case where different invariant measures exist, the long-time convergence of self-stabilizing diffusions has been addressed in e.g. \cite{Tamura1984}, \cite{AOP}, \cite{DuongTugaut2018}. 
  
\vspace{0.5cm}
  
\indent
From here on, the pairs \eqref{MV1}-\eqref{MV2} and \eqref{Ibis}-\eqref{IIbis} will be all subject to the following assumptions:
 
\noindent
$\mathbf{(A)}-(i)$ $V:\mathbb{R}^d\rightarrow \mathbb R$ is of class $\mathcal C^2$, uniformly convex at infinity, and such that $\nabla V$ is locally Lipschitz continuous and grows at most at a $2n$-polynomial rate. Namely, $\nabla V$ satisfies the following properties: for some threshold $R'>0$, the matrix $\inf_{||x||\ge R'}\nabla^2 V(x)$ is positive definite;
\[
\forall R>0,\:\:\sup_{\max(||x||,||y||)<R}\frac{||\nabla V(x)-\nabla V(y)||}{||x-y||}<\infty,
\]
and
\[
\sup_{x\in\mathbb R^d}\left\{(1+||x||^{2n})^{-1}||\nabla V(x)||\right\}<\infty\,,
\]
$||\cdot||$ denoting the Euclidean norm.\\
\noindent
$(\mathbf{A})-(ii)$ $V$ admits \emph{exactly} two distinct (strict) local minima located at the points $\lambda_1$ and~$\lambda_2$.

\noindent
$(\mathbf{A})-(iii)$ $F(x):=\frac{\alpha}{2}||x||^2$ with $\alpha>-\theta$ 
 for $\theta:=\inf_{x \in \bRb^d}\inf_{\xi \in \bRb^d\,:\,||\xi||=1}\big(\xi\nabla^2V(x)\xi \big)$.

\noindent
$(\mathbf{A})-(iv)$ Each initial condition $x_1$ and $x_2$ lies in a distinctive basin of attraction of $V$, that is: for $i=1$ or $2$, $x_i$ belongs to the set $\mathcal{G}(\lambda_i)$ defined by
\[
\mathcal{G}(\lambda_i):=\left\{z\in\mathbb R^d\,:\,\,\lambda_i=\lim_{t\rightarrow \infty}\phi_t(z)\,\text{ for }\,\phi_t(z)=z-\int_0^t\nabla V(\phi_s(z))\,ds\right\}\,.
\] 
Under Assumptions $(\mathbf{A})-(i)$ to $(\mathbf{A})-(iii)$, Equations~\eqref{MV1} and \eqref{MV2} are well-posed, with uniqueness holding in a path-wise sense, and the solutions $X$ and $Y$ each {\color{black}has} uniform-in-time finite moments of all orders (see \cite[Theorem 2.13]{HIP} for the precise statement{\color{black}s} and demonstrations of these results).  The same implication can be stated for \eqref{Ibis} and \eqref{IIbis} with: for all~$p>0$,
\[
\sup_{t,N}\mathbb E\left\{||X_t^{i,N}||^{2p}+||Y_t^{i,N}||^{2p}\right\}
<\infty\,.
\]
The condition $(\mathbf{A})-(iii)$ corresponds to the simplest form of synchronized regime and has been purposely chosen to simplify some proof arguments later on. (Possible extensions of our working assumptions, notably $(\mathbf{A})-(iii)$, will be presented at the end of this section.) In view of $(\mathbf{A})-(i)$, the self-stabilizing force compensates the lack of global convexity of the potential $V$, and the resulting \emph{effective potential} of the system,  
 $x\mapsto V(x)+\int F(x-z)\mu(dz) =V(x)+\frac{\alpha}{2}\int ||x-y||^2\mu(dy)$
  is uniformly 
   convex on $\mathbb R^d$, independently of the measure argument $\mu$.  
 
Without further assumptions and outside the one-dimensional case (which will be treated separately from the general $d$-dimensional case), tracking the first time $t$ where $X$ and $Y$ will collide is by nature ill-posed.  Indeed, in the case $d\ge 2$, as $B$ and $\tilde B$ almost surely  do not collide at finite time, under the assumptions $(\mathbf A)-(i)$ and  $(\mathbf A)-(iii)$ - and up to a change of probability measure - the same can be stated for the pairs $(X,Y)$ and $(X^{i,N},Y^{i,N})$. Properly, the first collision-time between \eqref{MV1} and \eqref{MV2} is obtained by the $\varepsilon\downarrow 0$-limit of the family of stopping times:
\begin{equation}
\label{berlin-bogota}
C_\varepsilon(\sigma):=\inf\left\{t\ge 0\,:\,||X_t-Y_t||\le 2\varepsilon\right\}\,,
\end{equation}
while the first collision-location is characterized by the $\varepsilon\downarrow 0$-limit of $(X_{C_\varepsilon(\sigma)},Y_{C_\varepsilon(\sigma)})$ in some region of the space. (Below, this limit will be often refer to as the {\it persistence} of the first collision-location.) 
Considering the continuously diffusive nature of $X$ and $Y$ for arbitrary $\sigma>0$, approximating true-collisions into $\epsilon$-collisions is rather natural. Equivalently, this approximation amounts to widening the point-materials $(X_t,Y_t)$ into a pair of moving permeable balls, with center of mass located in $X_t$ and $Y_t$ at each time $t$, and with a specified radius $\epsilon$ defining the radius of collision between the two bodies.
To avoid any trivial situation, the collision radius $\varepsilon$ has to be taken smaller than the smallest distance between the zero-noise limit $(\phi(x_1),\phi(x_2))$ of $(X,Y)$; that is
\begin{equation}\label{genthreshold}
\varepsilon_0:=
2^{-1}\inf_{t\ge 0}\{||\phi_t(x_1)-\phi_t(x_2) ||\}\,.
\end{equation}

The analog for \eqref{Ibis}-\eqref{IIbis} is characterized by the family of hitting times,
\begin{equation}
\label{berlin-bogota2}
C^i_{\varepsilon,N}(\sigma):=\inf\left\{t\ge 0\,:\,||X_t^{i,N}-Y_t^{i,N}||\le 2\varepsilon\right\}\,,\,1\leq i\le N\,.
\end{equation}

The assumptions $(\mathbf{A})-(i)$ to $(\mathbf{A})-(iv)$ are purposely set to generate a specific regime where collisions result from the sole and combined actions of the driving Brownians. As these actions vanish with $\sigma$, collision-time necessarily grows at a certain rate while the associated collision-location may remain in a balanced region between the wells $\lambda_1$ and $\lambda_2$. Both phenomena are intuitively determined by the parameters $\alpha$, $\sigma$, the radius $\varepsilon$ and the depth of the wells where lie the attractors $\lambda_1$ and $\lambda_2$. 

 More specifically, the {\color{black}combination} of $(\mathbf{A})-(i)$, $(\mathbf{A})-(ii)$ and $(\mathbf{A})-(iv)$ {\color{black}ensures} a bistable regime where all diffusions evolve in a landscape where $\lambda_1$ and $\lambda_2$ define two separate attractors, one for $X$ and the other one for $Y$. As $\sigma\downarrow 0$, all sources of randomness disappear and the paths of \eqref{MV1} and \eqref{MV2} naturally simplify into two gradient flows $\phi(x_1)$ and~$\phi(x_2)$ solutions to
 \[
 \frac{d\phi_t(x_k)}{dt}=-\nabla V(\phi_t(x_k)),\quad \phi_0(x_k)=x_k,\qquad k=1,2\,.
 \] 
  Due to $(\mathbf{A})-(iv)$, $\lim_{t\rightarrow\infty}\phi_t(x_k)=\lambda_k$ and, with $(\mathbf{A})-(ii)$, $\phi(x_1)$ and $\phi(x_2)$ are collision-free. In the same way, each moving dynamical couple of balls with a center of mass located at $\phi_t(x_1)$ and $\phi_t(x_2)$ are collision-free at all time $t$. Finite time $\varepsilon$-collisions so occur as long as the diffusive effects of $B$ and $\tilde B$ remain. As time goes by and as $\sigma$ vanishes, the  potential force prevails and, for a collision to happen, the driving Brownians have to force each diffusion $X$ and $Y$ to overcome their potential barrier. As such, and in view of $(\mathbf{A})-(ii)$ and $(\mathbf{A})-(iv)$,  the asymptotic $\lim_{\sigma\rightarrow 0}(X_{C_\varepsilon(\sigma)},Y_{C_\varepsilon(\sigma)})$ should remain in a region surrounding $\lambda_1$ and $\lambda_2$, and whose "width" depends on $\epsilon$.  This phenomenon is also expected for \eqref{Ibis} and \eqref{IIbis}. Indeed, as  $\sigma\downarrow 0$, the 
  particle system{\color{black}s} $(X^{i,N},Y^{i,N})$, $1\le i\le N$, {\color{black}converge} to the family of dynamical systems $(\phi^i(x_1),\phi^i(x_2))$, $1\le i\le N$, satisfying
 \[
 \frac{d\phi_t^i(x_k)}{dt}=-\nabla V(\phi_t^i(x_k))-\alpha\Big(\phi_t^i(x_k)-\frac 1{N}
 \sum_{j=1}^N \phi_t^j(x_k)\Big),\, \phi_0^i(x_k)=x_k,\,1\le i\le N,\, k=1,2\,.
 \]
As $(\mathbf{A})-(i)$ ensures that $\nabla V$ is locally Lipschitz continuous and as all the flows $\phi^{j}(x_k)$ start at $x_k$, a uniqueness argument yields that each $\phi^{i}(x_k)$ simply corresponds to $\phi(x_k)$.  

Our interest for the present study has been initially motivated with the modeling and the analysis of collisions induced by a vanishing random perturbation in swarming interacting multi-agent systems, notably in Cucker-Smale models. 
Introduced in \cite{CS07a}, \cite{CS07b},
these models broadly define a class of second order high dimensional systems representing, at each time, the position and the velocity 
of a finite population of individuals of the same specie. Albeit initially dispersed, individuals, through their interactions, adopt a common behavior over time. Since their introduction,  Cucker-Smale models and their connections with statistical physics have been extensively studied and adapted to wider situations in  social science and in economy (see e.g. \cite{NaldiPareTosc-10}, \cite{PareTosc-14}). The introduction of stochastic perturbation has also been considered where a variety of collective behaviors can be found depending on how the noise acts in the dynamic (additively or multiplicatively; privately or commonly; ...); see \cite{Pedeches2017}, \cite{CDL} and the references therein. In parallel to the impact of the noise in flocking models, another extension which motivated our setting is the introduction of leaders (see again \cite{PareTosc-14}) in the models, which influence the emergence of distinct agglomeration over distinct steady regions. Having this modeling perspective in mind, the particle systems~\eqref{Ibis} and \eqref{IIbis} - and their large population limit \eqref{MV1} and \eqref{MV2} - can be viewed as describing only the positions of two groups of bodies (representing e.g.  economical agents, animal populations or cells) where each group evolves independently from the other and is attracted to a specific source (e.g.  economical objectives, nutriment sources, or chemical attractants). Placed in a random environment where each individual is affected by an idiosyncratic noise, the two entities are forced to collide -  $\varepsilon$ being understood as a range of influence. In this framework, we specifically address the question of how characteristics of the first collision (time and location) behave as the source of the collision elapses.   
   
Compared to our original interest, the models \eqref{MV1}-\eqref{MV2} and \eqref{Ibis}-\eqref{IIbis} only provide a simplified version of our cases of interest: Langevin dynamics are eased into their over-damped~-~or Kramers-Smoluchowski - limits and all possible interaction, notably any possible post-collision effect, between the two self-stabilizing diffusions - or the two families of particles~-~are neglected.
Concretely, the extension of our main results to the framework of Cucker-Smale models - or even general second order dynamics - currently out-scopes the applicative and technical range of the present paper. This gap is inherent to the fact that LDP for Langevin dynamics have been scarcely
 addressed in the literature. This gap is further significant with  Freidlin-Wentzell's theory on exit-time problems - which carries the essence of this paper. Enabling to fill this gap would require to revisit Freidlin-Wentzell's theory and further thoroughly revisit the literature addressing this theory in the case of self-stabilizing diffusions.
 Despite this deviation from our initial intent, the present framework and the methodology developed later in the paper set a solid theoretical basis, solid enough to carry our scientific program in the near future. 
  
 \subsection{First collision-times viewed as first exit-times} 
  
 As it will become clear in a few lines, and following the terminology introduced in \cite{HIP}, the questions of estimating how fast the first collision-times between $(X,Y)$ and $(X^{i,N},Y^{i,N})$ grow and whether~-~and where - the first collision-locations $(X_{C_\varepsilon(\sigma)},Y_{C_\varepsilon(\sigma)})$ and $(X^{i,N}_{C^i_{\varepsilon,N}(\sigma)},Y^{i,N}_{C^i_{\varepsilon,N}(\sigma)})$ may persist amount to establishing a Kramers' type law for the systems \eqref{MV1}-\eqref{MV2} and \eqref{Ibis}-\eqref{IIbis}. By definition,~$C_\varepsilon(\sigma)$ can be alternatively viewed as the first time when the diffusion $(X,Y)$ enters the domain $\triangle_\varepsilon:=\{(x,y)\in\mathbb R^d\,:\,||x-y||\le 2\varepsilon\}$ or equivalently the first time when $(X,Y)$ leaves~$\mathbb R^{2d}\setminus \triangle_\varepsilon$.
 With this view, the natural framework to study our question is the one of the large deviations, and more precisely the Freidlin-Wentzell theory on the exit-problem for the perturbed dynamical systems out of stable sets (see the definition below). While we refer the reader to \cite[Chapters 3 and 4]{FW98} - and \cite[{\color{black}Sections 5.6, 5.7}]{DZ} - for a detailed introduction on this theory, for the sake of completeness, we briefly recollect essential results which will be used later on. 

Consider the diffusion process:
\begin{equation}\label{GenSGF}
z_t^\sigma=z_0+\sigma \mathcal W_t+\int_0^t b\left(z_s^\sigma\right)ds\,,\,t\ge 0,
\end{equation}
where $b$ is a smooth vector field on $\bRb^m$ ($m\geq1$) and $\mathcal W$ is a $\mathbb R^m$-Brownian motion. As~$\sigma$ decreases, the paths of $z^\sigma$ become closer to the deterministic dynamical system $\Psi(z_0)$ defined by 
\begin{equation*}\label{GenSGFbis}
\Psi_t(z_0)=z_0+\int_0^t b\left(\Psi_s(z_0)\right)ds,\,t\ge 0.
\end{equation*}
Precisely, the  way at which $z^\sigma$ approaches $\Psi(z_0)$ obeys to the following large deviations principle: for any finite arbitrary time horizon $T$ and for any $\delta>0$, as $\sigma\downarrow 0$,
\begin{equation}\label{LDP}
\log\PP\left\{\sup_{t\in[0;T]}\gaga z_t^\sigma-\Psi_t(z_0)\drdr>\delta\right\}
\approx -\frac{1}{2\sigma^2}\inf_{\Phi}\left\{\int_0^T\gaga \frac{d\Phi_t(z_0)}{dt}-b(\Phi_t(z_0))\drdr^2\,dt\right\}\,,
\end{equation}
the infimum being taken over the class of continuously differentiable functions  $\Phi:[0,T]\rightarrow \mathbb R^m$ starting from $z_0$ at $t=0$, and such that $\max_{0\le t\le T}\gaga \Phi_t(z_0)-\Psi_t(x_0)\drdr>\delta$. In the case where the orbits $\Psi(z_0)$ have a unique attractor - that is $\lim_{t\rightarrow \infty}\Psi_t(z_0)=a_0$ for any starting point $z_0$ -
the paths of $z^\sigma$ are naturally wandering around a neighborhood of $a_0$, as $\sigma$ decreases to $0$. 

As such, the diffusion will stick to a neighborhood of $a_0$  for an arbitrary long-time, for small values of $\sigma$. One may wonder which time scale will it take for $z^\sigma$ to escape such neighborhood. For this question, the relevant sets to consider are given by the class of stable sets.

\begin{defn}
\label{def:stab} A subset $\mathcal{G}$ of $\mathbb{R}^{d}$ is said to be stable by the vector field $b$ if the orbits~$\Psi(z_0)=(\Psi_t(z_0))_{t\ge 0}$ defined as in~\eqref{GenSGFbis} are included in $\mathcal{G}$ for all initial state $z_0\in\mathcal{G}$.
\end{defn}
While the terminology ``stable by'' is more often referred to as ``positively invariant by'' in the literature, we retain the former terminology to remain consistent with the bibliographic sources referenced hereafter.

Under not-too restrictive assumptions on $b$,  M. I. Freidlin and A. D. Wentzell established the exponential growth of the exit-time of $z^\sigma$ from a domain $\cGc$ stable by $b$  and the concentration point of $z^\sigma$ evaluated at this exit-time.
 In the special case of stochastic gradient flows, with $b=-\nabla U$, Freidlin-Wentzell results formulate as follows (see \cite[Chapter 4, Theorem~3.1]{FW98}, \cite[Theorem 5.7.11]{DZ}):

\begin{thm}\label{thm:KramersDZ} Assume that $U:\mathbb R^m\rightarrow \mathbb R$ is of class $\mathcal C^1$ with $\nabla U$ Lipschitz continuous on $\mathbb R^d$, and define the exit-time
$$
\tau_{\cGc}(\sigma)=\inf\{t>0\,:\,z_t^\sigma \in\partial \cGc\},
$$
for $z^\sigma$ as in \eqref{GenSGF} and $\mathcal G$ an open bounded set of $\mathbb R^d$, stable by $-\nabla U$ and $\mathcal G$ contains one and only one minimizer of $U$, $a_0$.
If, for all $z_0$ in the closure $\overline{\mathcal G}$, $\Psi_t(z_0)$ converges to $a_0$ as $t\uparrow \infty$ then, for any $z_0$ in $\mathcal{G}$, we have
\begin{equation}\label{GenericKramers}
\lim_{\sigma\to0}\PP\left\{\exp\left[\frac{2}{\sigma^2}\left(\underline{H}-\delta\right)\right]<\tau_{\cGc}(\sigma)<
\exp\left[\frac{2}{\sigma^2}\left(\underline{H}+\delta\right)\right]\right\}=1\,,
\end{equation}
where
$$
\underline{H}:=\inf_{z\in \partial\mathcal G}\big(U(z)-U(a_0)\big)\,,
$$
corresponds to the so-called {\emph exit-cost} (from $\mathcal G$).
\noindent
Additionally, 

\noindent
$(a)$ For all $z_0$ in $\mathcal G$, $\lim_{\sigma\rightarrow 0}\frac{\sigma^2}{2}\log\Big(\mathbb E\big\{\tau_{\cGc}(\sigma)\big\}\Big)=\underline{H}$\,;

\noindent
$(b)$ For any closed subset $\mathcal N$ of $\partial \mathcal G$ such that

\begin{equation*}
\inf_{z\in \mathcal N}\left[U(z)-U(a_0)\right] >\inf_{z\in\partial \mathcal G}\left[U(z)-U(a_0)\right]\,,
\end{equation*}

then
\[
\lim_{\sigma\to0}\mathbb P\left\{z_{\tau_{\cGc}(\sigma)}^\sigma\in \mathcal N\right\}=0.
\]
\end{thm}
The properties \eqref{GenericKramers} and $(a)$ above state that 
$\tau_{\mathcal G}(\sigma)$ grows at an exponential rate  in probability and in average. The property $(a)$ specifically recovers the so-called Arrhenius' law (\cite{Arrhenius-1889}, \cite{Laider-84}) also known as a first-form of Kramers-Eyring law (\cite{Eyring-35}, \cite{Kramers-40}, \cite{Berglund-13}). Additionally, the result demonstrates that the exit-location $z^\sigma_{\tau_{\mathcal G}(\sigma)}$ asymptotically concentrates on the region where it is the least costly  to exit the domain starting from the attraction point $a_0$. In particular, if there exists a unique~$z^\star$ in $\partial\mathcal G$ such that $U(z^\star)-U(a_0)=\inf_{z\in\partial\mathcal G}(U(z)-U(a_0))$, then, for all $\delta>0$, $z_0\in\mathcal G$,
\begin{equation*} 
\lim_{\sigma\rightarrow 0}\mathbb P\left\{||z_{\tau_{\cGc}(\sigma)}^\sigma-z^\star||<\delta\right\}=1\,.
\end{equation*}
It is worth noticing that Freidlin-Wentzell results are purposely stated here in a simplified framework and the original  Freidlin-Wentzell result in  \cite{FW98}, \cite{DZ} also holds true  for non-reversible processes. On the other hand, the force field $-\nabla U$ is assumed to be globally Lipschitz, whereas, in the paper, we consider a setting which clearly violates this condition. In effect, the globally Lipschitz assumption can be weakened to a local one (this extension was previously remarked in e.g.  \cite{HIP}).

The question of whether Freidlin-Wentzell theory applies to self-stabilizing diffusions is not new and has been addressed by S.~Herrmann, P.~Imkeller and D. Peithmann in their work~\cite{HIP}. Therein, the authors establish an analog of Theorem~\ref{thm:KramersDZ}, which they refer to as a Kramers' type law, for the self-stabilizing diffusion $Z$ satisfying
\begin{equation*}
\left\{
\begin{aligned}
&Z_t=z+\sigma B_t-\int_0^t\Big(\nabla U\left(Z_s\right)+\int \nabla F\left(Z_s-x\right)\,\mu^Z(s,dx)\Big)ds\,,\\
&\mu^Z(t)=\text{Law}(Z_t)\,,\,t\ge 0\,.
\end{aligned}
\right.
\end{equation*}
The potentials $U$ and $F$ are both assumed to be uniformly convex and the unique attractor of~$U$ is $a$. The law for the exit-time $\tau_{\mathcal G}(\sigma):=\inf\{t\ge 0\,:\,Z_t\notin \mathcal G\}$ is there (\cite[Theorem~4.2 and Section 5]{HIP}) given by  
\begin{equation}\label{KramersSelfStab}
\tau_{\mathcal G}(\sigma)\approx\exp\bigg[\frac {2}{\sigma^2}\underline H \bigg],\ \ \underline{H}:=\inf_{z\in\partial\mathcal{G}}\{U(z)-U(a)+F(z-a)\}\,,
\end{equation}
for $\approx$ denoting asymptotic equivalence as in \eqref{GenericKramers}.

In a series of papers, the second author has successfully extended these Kramers' type laws to the situation of a non-globally convex potential $U$, under  a synchronized regime or weaker assumptions (see~\cite{JOTP},~\cite{Alea} and references therein) as well as for stochastic particle systems (\cite{ESAIM_particles}). Parts of the strategies displayed in these papers will be adapted or extended herein. Notably, the coupling technique which asserts that a self-stabilizing diffusion can be found arbitrarily close to a given Markov process after a certain (deterministic) time.  

\subsection{Main results}
For the case of the self-stabilizing diffusions \eqref{MV1}-\eqref{MV2}, the Kramers' type law for $C_\varepsilon(\sigma)$ defined in \eqref{berlin-bogota} states as follows:
\begin{thm}
\label{thm:main1} Let $\underline{H}_0$ be the minimum of the function $H_0:\mathbb R^d\rightarrow \mathbb R$ given by
\begin{equation}\label{collisioncost}
H_0(\lambda):=2V(\lambda)-V(\lambda_1)-V(\lambda_2)+F(\lambda-\lambda_1)+F(\lambda-\lambda_2)\,,
\end{equation}
and let $\lambda_0$ be the unique minimizer of $H_0$.
Then, for any $\delta>0$, 

\begin{equation*}
\lim_{\varepsilon\to0}\lim_{\sigma\to0}\PP
\left\{\exp\left[\frac{2}{\sigma^2}\left(\underline{H}_0-\delta\right)\right]
<{C}_\varepsilon(\sigma)
<\exp\left[\frac{2}{\sigma^2}\left(\underline{H}_0+\delta\right)\right]\right\}=1\,.
\end{equation*}
Moreover, the collision-location $(X_{C_{\varepsilon}(\sigma)},Y_{C_\varepsilon(\sigma)})$ persists asymptotically in the vicinity of $\lambda_0$. Namely: for any $\delta>0$,
\begin{equation*}
\lim_{\varepsilon\to0}\lim_{\sigma\to0}\PP\left\{\max\bigg(\gaga X_{{C}_\varepsilon(\sigma)}-\lambda_0\drdr,\gaga Y_{{C}_\varepsilon(\sigma)}-\lambda_0\drdr\bigg)\leq\delta\right\}=1\,.
\end{equation*}
\end{thm}

\begin{rem} Notice that the existence and the uniqueness of the minimizer $\lambda_0$ is a direct consequence of the {\it synchronization} assumption $({\bf A})-(iii)$. As this condition yields the uniform 
convexity of $x\mapsto V(x)+F(x-m)$ for any $m$ in $\mathbb R^d$, readily,
\begin{equation*}\label{collisionlocation}
\lambda_0=\Big(\nabla V+\alpha {\rm Id}\Big)^{-1}\bigg(\frac{\alpha}{2}(\lambda_1+\lambda_2)\bigg),
\end{equation*}
for ${\rm Id}$ denoting the identity map $x\mapsto {\rm Id}(x)=x$. 
Illustratively, in the situation of a perfectly symmetrical landscape - namely $\lambda_1=-\lambda_2$ - $\lambda_0$ is simply the root of $x\mapsto \nabla V(x)+\alpha x$, and in the more prototypical case where $V(x)=x^4/4-x^2/2$, the first collision-location persists at the hill $\lambda_0=0$. In the situation of the asymmetric double-wells potential $V(x)=x^4/4+x^3/3-x^2/2$, where the wells are located at the points $-1/2\pm \sqrt{5}/2$ - the lowest well lying in $-1/2- \sqrt{5}/2$ -, and for a synchronization $\alpha>4/3$, $\lambda_0$ corresponds to the root of the polynomial $V'(x)+\alpha (x+1/2)$. This point is notably distinct from the saddle point~$x=0$ of~$V'$.
\end{rem}

{The above Kramers' type law is on par with what one would have expected from a heuristic application of~\eqref{KramersSelfStab}: assuming that the domain~$(\mathbb R^d\times\mathbb R^d)\setminus\triangle_\varepsilon$ was stable, $C_\varepsilon(\sigma)$ would obey to the general asymptotic estimate \eqref{GenericKramers} with the exit-cost
\[
\inf_{(x,y)\in\partial \triangle_\varepsilon}\Big(V(x)-V(\lambda_1)+F(x-\lambda_1)+V(y)-V(\lambda_2)+F(y-\lambda_2)\Big)\,.
\]
Also, all possible exit-locations from $(\mathbb R^d\times\mathbb R^d)\setminus\triangle_\varepsilon$ would be resting in the corresponding set of minimizers. Since the boundary $\partial \triangle_\varepsilon$ corresponds to the set~$\{(x,y)\in \mathbb R^d\times \mathbb R^d\,:\,||x-y||=2\varepsilon\}$, as $\varepsilon$ is taken smaller and smaller, the exit-cost  would get closer to  
\[
\inf_{x}\Big(2V(x)-V(\lambda_1)-V(\lambda_2)+F(x-\lambda_1)+F(x-\lambda_2)\Big)\,,
\]
 that is $\underline{H}_0=\inf H_0$, and  exit-locations closer to the related minimizer, $\lambda_0$. However, due to the {\it a-priori} lack of stability of $(\mathbb R^d\times\mathbb R^d)\setminus\triangle_\varepsilon$, this direct argument can not be applied and a rigorous demonstration of Theorem \ref{thm:main1} requires a substantial detour. To circumvent the possible instability of $(\mathbb R^d\times\mathbb R^d)\setminus\triangle_\varepsilon$, we rather interpret $C_\varepsilon(\sigma)$ as the first time $X$ and~$Y$ are simultaneously found at an $\varepsilon$-neighborhood of any point $\lambda$ of $\mathbb R^d$. Precisely~$C_\varepsilon(\sigma)$ can be written as $\inf_{\lambda}\beta_{\lambda,\varepsilon}(\sigma)$ for $\beta_{\lambda,\varepsilon}(\sigma)$ the first time $X$ and $Y$ both enter the ball of radius $\varepsilon$ and centered in $\lambda$. Indeed, by triangular inequality, $C_\varepsilon(\sigma)\leq\inf_{\lambda}\beta_{\lambda,\varepsilon}(\sigma)$, meanwhile $C_\varepsilon(\sigma)\geq\beta_{\lambda_{\varepsilon,\sigma},\varepsilon}(\sigma)\geq\inf_{\lambda}\beta_{\lambda,\varepsilon}(\sigma)$, for $\lambda_{\varepsilon,\sigma}:=2^{-1}(X_{C_\epsilon(\sigma)}+Y_{C_\epsilon(\sigma)})$.
 As each~$\beta_{\lambda,\varepsilon}(\sigma)$ approximates the first time $X$ and $Y$ meet at the point $\lambda$,  these stopping times allow a parametrization of the possible collision-location (for instance at $\lambda$). Subsequently, this parametrization enables us to lean against Freidlin-Wentzell theory, provided some suitable preliminaries (which will be the subject of Section \ref{sec:LinearCase}), and next, to borrow and to adapt the coupling techniques from~\cite{Alea},~\cite{Kinetic}.

This strategy further allows to derive similar asymptotics for~$C^i_{\varepsilon,N}(\sigma)$ defined in \eqref{berlin-bogota2}. The Kramers' type law in this case is given by the following theorem. 
 
 \begin{thm}\label{thm:main2} Let $\underline{H}_0$ and $\lambda_0$ be as in Theorem \ref{thm:main1}. Then, for $N$ large enough, it holds: for any $1\le i\le N$, $\delta>0$,
\begin{equation*}
\lim_{\varepsilon\to0}\lim_{\sigma\to0}\PP\left\{\exp\left[\frac{2}{\sigma^2}\left(\underline{H}_0-\delta\right)\right]<\mathcal{C}^i_{\varepsilon,N}(\sigma)<
\exp\left[\frac{2}{\sigma^2}\left(\underline{H}_0+\delta\right)\right]\right\}=1\,,
\end{equation*}
 and
\begin{align*}
&\lim_{\varepsilon\to0}\lim_{\sigma\to0}\PP\left\{\max\bigg(\gaga X_{\mathcal{C}^i_{\varepsilon,N}(\sigma)}^{i,N}-\lambda_0\drdr,\gaga Y_{\mathcal{C}^i_{\varepsilon,N}(\sigma)}^{i,N}-\lambda_0\drdr\bigg)\leq\delta\right\}=1\,.\nonumber
\end{align*}
\end{thm}

\subsection{Organization of the paper} The next section serves as a preliminary step as well as a guideline for treating self-stabilizing and particle diffusions. In that section, we focus on establishing a Kramers' type law for the first collision-time of two stochastic gradient flows driven each by a different uniformly 
convex potential. Following this preliminary, and, as previously mentioned, relying on a coupling argument, we demonstrate Theorem~\ref{thm:main1} in Section~\ref{sec:SelfStabilizingCase}, and  Theorem~\ref{thm:main2} in Section \ref{sec:ParticleCase}. The last section, Section \ref{sec:1DCase}, focuses on the one-dimensional situation where the exact first collision-times 
\[
C(\sigma)=\inf\{t\ge 0\,:\,X_t=Y_t\}, \ \ \ C^i_N(\sigma)=\inf\{t\ge 0\,:\,X^{i,N}_t=Y^{i,N}_t\}\,,
\]
can be defined. Analog for Theorems~\ref{thm:main1} and~\ref{thm:main2} (see Theorems~\ref{victoire2} and~\ref{victoire3}) are established with more direct arguments than in the multi-dimensional case.

\subsection{Discussion on some extensions}

As previously mentioned, our main results are presented in the simplest form of self-stabilization with $F$ being given by the quadratic form~$\frac{\alpha}{2}||x||^2$ - with $\alpha$ satisfying the synchronization assumption $(\mathbf{A})-(iii)$. The only crucial assumption needed for Theorems~\ref{thm:main1} and \ref{thm:main2} is the consequence of this synchronization condition which makes $x\mapsto V(x)+\int F(x-y)\,\nu(dy)$ uniformly  
convex, independently of $\nu$. The nonlinear McKean derive can be extended into the more general form~$-\int \nabla F(x-y)\mu(t,dy)$, provided this convexity property holds. This could be achieved for~$F(x)=G(||x||)$ for an even polynomial function $G$, with a degree larger than 2, and such that $G(0)=0$. This setting has already been considered in~\cite{ESAIM_particles}, and, with not too much effort, the coupling results, Lemma~\ref{lucille} and  Proposition~\ref{lucille2}, can be established in this weaker setting.

In addition to the extension of the interaction potential $F$, our setting  can also be extended to the situation where \eqref{MV1}-\eqref{MV2} start from random initial states. As long as $(X_0,Y_0)$ is a.s. bounded (to ensure uniform moments control) and as long as the marginal laws of $X_0$ and $Y_0$ have full support on different basins of attraction of $V$, our main results still hold true.

The condition $\mathbf{(A)}-(ii)$ can also be weakened to take into account  a multi-wells landscape. This means to consider, in place of $\mathbf{(A)}-(ii)$, that $V$ admits $m$ (with $m> 2$) distinct minimizers located at distinct points, $\lambda_1,\cdots,\lambda_m$. This generalization does not fundamentally alter the collision between two (self-stabilizing) diffusions and potentially opens the door to consider collision between multiple diffusions. 
Heuristically, we expect that analogs to  Theorems ~\ref{thm:main1} and~\ref{thm:main2}  should hold for the approximated first collision-time between the $m$ self-stabilizing diffusions related to the family of wells. The corresponding exit-cost should be
\[
\inf_\lambda\left\{ \sum_{k=1}^m\big(V(\lambda)-V(\lambda_k)+F(\lambda-\lambda_k)\big)\right\}\,,
\] 
and the collision-location should be found at the point:
\[
\Big(\sum_{l=1}^m\nabla\Psi_l\Big)^{-1}(0),\qquad \Psi_l(x):=V(x)+F( x-\lambda_l)\,.
\]
Assuming back $(\mathbf{A})-(ii)$, this would mean
\[
\lambda_0=\Big(\nabla V+\alpha {\rm Id}\Big)^{-1}\bigg(\alpha m^{-1}\sum_{l=1}^m\lambda_{l}\bigg)\,.
\] 
The rigorous derivation of these heuristics are nonetheless non-trivial, and should be addressed carefully. 
 
Finally, coming back to the assumption $\mathbf{(A)}-(iii)$, the {\it synchronization} can itself be weakened. While this condition has been essential in our proof arguments, as pointed out in \cite[Corollary D]{Alea} (for $d=1$) and \cite[Theorem 3.4]{Kinetic} (for general $d>1$), {\it synchronization} may also be weakened for coupling techniques albeit for the case where $\nabla F$ is linear (i.e. $F(x)=||x||^2/2$). The weaker condition formulates there as: for $i=1,2$, there exists $\rho_i>0$ such that, for $x\in\mathbb R^d$
\begin{equation*}
\left( x-\lambda_i\right)\left(\nabla V(x)+\alpha(x-\lambda_i)\right)\geq\rho_i||x-\lambda_i||^2\,.
\end{equation*} 
Let us point out that this condition allows broadly a control of the proximity between
the law of the self-stabilizing diffusions and their assigned attractors (see again \cite{Alea} and \cite{Kinetic} for the precise statement). From this control,~\eqref{carl_bis} below may still deduced. However, coupling estimates will cease to hold true and obtaining Theorems~\ref{thm:main1} and~\ref{thm:main2} under this weaker condition will necessitate a complete new strategy, rather based on expanding~\cite{HIP} into non-(global) convex.
 
\section{On the first collision of two independent stochastic gradient flows}\label{sec:LinearCase}

In this section, we establish the zero-noise asymptotic of the \emph{approximated} collision-time 
\begin{equation}
\label{lori}
c_\epsilon(\sigma):=\inf\left\{t\ge0\,\,:\,\,\gaga x^\sigma_t-y^\sigma_t\drdr \leq2\epsilon\right\},\,\epsilon>0,
\end{equation}
related to the two general  stochastic gradient flows:
\begin{subequations}
\label{elf}
\begin{equation}
\label{carole}
x^\sigma_t=x_0+\sigma B_t-\int_0^t\nabla\Psi_1\left(x^\sigma_s\right)ds,\,\,\,t\ge 0,
\end{equation}
and
\begin{equation}
\label{daryl}
y^\sigma_t=y_0+\sigma \widetilde{B}_t-\int_0^t\nabla\Psi_2\left(y^\sigma_s\right)ds\,,\,\,\,t\ge 0.
\end{equation}
\end{subequations}
The starting points, $x_0$ and $y_0$, are assumed to be distinct,  and the driving potentials, $\Psi_1$ and $\Psi_2$, to be uniformly 
 convex, of class $\mathcal C^2$, and to achieve their minimum in two different points,~$\lambda_1$ and $\lambda_2$ respectively. Additionally, by analogy with $(\mathbf{A})-(iv)$, we assume the orbits: 
\[
\varphi_t^{1,-}=x_0-\int_0^t\nabla\Psi_1\left(\varphi_s^{1,-}\right)ds,\,\,\,t\ge 0,\
\]
and
\[
\varphi_t^{2,-}=y_0-\int_0^t\nabla\Psi_2\left(\varphi_s^{2,-}\right)ds,\,\,\,t\ge 0,
\]
never hit each other at any time $t$. This assumption does not impose the graphs $\Psi_1$ and $\Psi_2$ to be disjoints but ensures that $x^\sigma$ and $y^\sigma$ are collision-free at $\sigma=0$.

Following this last assumption, we set the radius $\epsilon$ in \eqref{lori} to be strictly smaller than
\begin{equation}\label{Linear:Init}
\epsilon_0:=2^{-1}\inf_{t\geq0}\gaga\varphi_t^{1,-}-\varphi_t^{2,-}\drdr> 0\,.
\end{equation}
This way, the initial states, $x_0$ and $y_0$, and the attracting points, $\lambda_1$ and $\lambda_2$, are separated by a distance strictly larger than $2\epsilon_0$. We finally assume from here on that $\epsilon<\epsilon_0$, making as such $c_\epsilon(\sigma)$ non-trivial. (If $\epsilon$ was chosen larger than $\epsilon_0$, then, in view of the LDP~\eqref{LDP}, there would exist $\sigma_0>0,\,T_0<\infty$ such that $\mathbb P\{c_\epsilon(\sigma)\le T_0\}=1$ for any $\sigma\le \sigma_0$ and $c_\epsilon(\sigma)$ would be bounded a.s..)

The asymptotic of $c_\epsilon(\sigma)$ is brought by the interpretation 
$c_\epsilon(\sigma)=\inf_{\lambda\in\bRb^d}\tau_{\lambda,\epsilon}(\sigma)$
for $\tau_{\lambda,\epsilon}$ defining the first time $x^\sigma$ and $y^\sigma$ are simultaneously located at a $\epsilon$-neighborhood of a given point $\lambda$ of $\mathbb R^d$; that is
\begin{equation}
\label{andrea}
\tau_{\lambda,\epsilon}(\sigma):=\inf\left\{t\ge 0\,\,:\,\,||x^\sigma_t-\lambda||\leq\epsilon\mbox{ and }||y^\sigma_t-\lambda||\leq\epsilon\right\}\,.
\end{equation}
From this interpretation, below we first establish a Kramers' type law for the time $\tau_{\lambda,\epsilon}(\sigma)$ and for the location~$(x^\sigma_{\tau_{\epsilon,\lambda}(\sigma)},y^\sigma_{\tau_{\lambda,\epsilon}(\sigma)})$ (see Lemma~\ref{dale-bis} below), deduce next a Kramers' type law for $c_\epsilon(\sigma)$ and~$(x^\sigma_{c_\epsilon(\sigma)},y^\sigma_{c_\epsilon(\sigma)})$ for small enough $\epsilon>0$ (Proposition~\ref{lacollision}), and finally conclude on the asymptotic behavior at the limit $\epsilon\downarrow 0$ (Theorem~\ref{lacollisionthm}).  
 
\subsection{Asymptotic estimates for $\tau_{\lambda,\epsilon}(\sigma)$}
\label{subsec:Linear-collisionA}
 
  According to its very definition,
 $\tau_{\lambda,\epsilon}(\sigma)$ 
 corresponds to the first time the diffusion $(x^\sigma,y^\sigma)$ enters $\overline{\mathbb{B}(\lambda;\epsilon)}\times\overline{\mathbb{B}(\lambda;\epsilon)}$, or equivalently exits from the domain
$$
\left(\bRb^d\times\bRb^d\right)\setminus\left(\overline{\mathbb{B}(\lambda;\epsilon)}\times\overline{\mathbb{B}(\lambda;\epsilon)}\right)=:\left(\overline{\mathbb{B}(\lambda;\epsilon)}\times\overline{\mathbb{B}(\lambda;\epsilon)}\right)^c\,,
$$
for $\overline{\mathbb{B}(\lambda;\epsilon)}$ denoting the closed ball centered in $\lambda$ and of radius $\epsilon$. Considering that the set~$\left(\bRb^d\times\bRb^d\right)\setminus\left(\overline{\mathbb{B}(\lambda;\epsilon)}\times\overline{\mathbb{B}(\lambda;\epsilon)}\right)$ is not necessarily stable by $(-\nabla\Psi_1,-\nabla\Psi_2)$, Theorem~\ref{thm:KramersDZ} can not yet be applied to deduce the asymptotics of $\tau_{\lambda,\epsilon}(\sigma)$. This technical difficulty can be bypassed by a two-steps modification of the exit-set~$\left(\overline{\mathbb{B}(\lambda;\epsilon)}\times\overline{\mathbb{B}(\lambda;\epsilon)}\right)^c$ to make it suitable for applying Theorem \ref{thm:KramersDZ}.  
  
 As a first modification, let us consider the sets 
 \begin{equation*}
\label{otis1}
\mathcal{D}_{\lambda,\epsilon}^1:={\left\{\varphi_t^{1,+}(x)\,\,:\,\,t\ge 0,\,x\in\overline{\mathbb{B}\left(\lambda;\epsilon\right)}\right\}}\,,
\end{equation*}
and
\begin{equation*}
\label{otis2}
\mathcal{D}_{\lambda,\epsilon}^2:={\left\{\varphi_t^{2,+}(y)\,\,:\,\,t\ge 0,\,y\in\overline{\mathbb{B}\left(\lambda;\epsilon\right)}\right\}}\,,
\end{equation*}
for $\varphi^{1,+}(x)$ and $\varphi^{2,+}(y)$ corresponding to the {\it ascending} flows related to $\nabla \Psi_1$ and $\nabla\Psi_2$:
\begin{equation*}
\varphi_t^{1,+}(x)=x+\int_0^t\nabla\Psi_1\left(\varphi_s^{1,+}(x)\right)ds,\,t\ge 0 
\,,
\end{equation*}
and
\begin{equation*}
\varphi_t^{2,+}(y)=y+\int_0^t\nabla\Psi_2\left(\varphi_s^{2,+}(y)\right)ds,\,t\ge 0
\,.
\end{equation*}
By definition, $\mathcal{D}_{\lambda,\epsilon}^i$, for $i=1$ or $i=2$, defines a closed domain containing the set of points attainable by the flow $\varphi^{i,+}$ starting from $\overline{\mathbb{B}(\lambda;\epsilon)}$.
Subsequently, the complementary~$\mathbb R^d\setminus\mathcal{D}_{\lambda,\epsilon}^i$ corresponds to the largest set stable by $-\nabla\Psi_i$ contained in $\overline{\mathbb{B}(\lambda;\epsilon)}^c$.
The inherent stability properties further guarantee that
 the domain $\left(\bRb^d\times \bRb^d\right)\setminus(\mathcal{D}_{\lambda,\epsilon}^1\times\mathcal{D}_{\lambda,\epsilon}^2)$ 
is stable by~$(-\nabla\Psi_1,-\nabla\Psi_2)$.  This statement can be simply checked as follows: for any $(x',y')$ in~$\left(\bRb^d\times \bRb^d\right)\setminus(\mathcal{D}_{\lambda,\epsilon}^1\times\mathcal{D}_{\lambda,\epsilon}^2)$, either $x'$ lies in $\bRb^d\setminus\mathcal{D}_{\lambda,\epsilon}^1$ or $y'$ lies in $\bRb^d\setminus\mathcal{D}_{\lambda,\epsilon}^2$, and, in each case, at least one of the marginal domain is stable. Since $\left(\bRb^d\times \bRb^d\right)\setminus(\mathcal{D}_{\lambda,\epsilon}^1\times\mathcal{D}_{\lambda,\epsilon}^2)$ can be rewritten as the union of stable sets $\left(\bRb^d\times\left(\mathcal{D}_{\lambda,\epsilon}^2\right)^c\right)\bigcup\left(\left(\mathcal{D}_{\lambda,\epsilon}^1\right)^c\times\bRb^d\right)$, where~$\left(\mathcal{D}_{\lambda,\epsilon}^i\right)^c:=\bRb^d\setminus\mathcal{D}_{\lambda,\epsilon}^i$, the domain is necessarily stable by $(-\nabla \Psi_1,-\nabla \Psi_2)$.

For any $\lambda\in\mathbb R^d$ and $\epsilon>0$, while $\mathcal{D}_{\lambda,\epsilon}^1\times \mathcal{D}_{\lambda,\epsilon}^2$ is larger than $\mathbb{B}(\lambda;\epsilon)\times \mathbb{B}(\lambda;\epsilon)$, the complementary of this domain fulfills the condition of Theorem \ref{thm:KramersDZ} and exit-costs are only computed along $\partial (\mathbb B(\lambda;\epsilon)\times \mathbb B(\lambda;\epsilon))$, in two main situations: $(a)$ when $\lambda$ is at a distance strictly larger than $\epsilon$ from both wells; $(b)$ when $\lambda$ lies in a $\epsilon$-neighborhood of one of the two wells.

$\bullet$ For $(a)$: whenever $\min_{i=1,2}(||\lambda-\lambda_i||)>\epsilon$, the ball $\mathbb B(\lambda;\epsilon)$ and the domain $\left(\bRb^d\times \bRb^d\right)\setminus(\mathcal{D}_{\lambda,\epsilon}^1\times\mathcal{D}_{\lambda,\epsilon}^2)$ do not contain any well.  According to Theorem \ref{thm:KramersDZ}, the exit-cost related to~$(x^\sigma,y^\sigma)$ leaving the domain is given by
\[
\inf_{(x,y)\in \partial (\mathcal{D}_{\lambda,\epsilon}^1\times\mathcal{D}_{\lambda,\epsilon}^2)} \big(\Psi_1(x)-\Psi_1(\lambda_1)+\Psi_2(y)-\Psi_2(\lambda_2)\big).
\]
Since the minima of $\Psi_1$ and $\Psi_2$ are located outside $\mathbb B(\lambda;\epsilon)$ - and so are outside $\mathcal{D}_{\lambda,\epsilon}^1$ and~$\mathcal{D}_{\lambda,\epsilon}^2$~-~the minimum of $\Psi_1$ and the one of $\Psi_2$ lie necessarily outside $\mathcal{D}_{\lambda,\epsilon}^1\times\mathcal{D}_{\lambda,\epsilon}^2$. Therefore
\begin{align*}
&\inf_{(x,y)\in \partial (\mathcal{D}_{\lambda,\epsilon}^1\times\mathcal{D}_{\lambda,\epsilon}^2)} \big(\Psi_1(x)-\Psi_1(\lambda_1)+\Psi_2(y)-\Psi_2(\lambda_2)\big)\\
&=\inf_{(x,y)\in \mathcal{D}_{\lambda,\epsilon}^1\times\mathcal{D}_{\lambda,\epsilon}^2} \big(\Psi_1(x)-\Psi_1(\lambda_1)+\Psi_2(y)-\Psi_2(\lambda_2)\big)\\
&=\inf_{x\in \mathcal{D}_{\lambda,\epsilon}^1} \big(\Psi_1(x)-\Psi_1(\lambda_1)\big)+\inf_{y\in \mathcal{D}_{\lambda,\epsilon}^2} \big(\Psi_2(y)-\Psi_2(\lambda_2)\big)\\
&=\inf_{x\in \partial\mathcal{D}_{\lambda,\epsilon}^1} \big(\Psi_1(x)-\Psi_1(\lambda_1)\big)+\inf_{y\in \partial\mathcal{D}_{\lambda,\epsilon}^2} \big(\Psi_2(y)-\Psi_2(\lambda_2)\big)\,.
\end{align*}
Additionally, we can observe that, for $i\in\{1,2\}$, the quantity $\inf_{x\in \partial \mathcal{D}_{\lambda,\epsilon}^i}\big(\Psi_i(x)-\Psi_i(\lambda_i)\big)$ is identical to $\inf_{x\in\mathbb \partial\overline{\mathbb{B}(\lambda;\epsilon)}}\big(\Psi_i(x)-\Psi_i(\lambda_i))$. This assertion can be checked, on one side, by observing that, as $\lambda_i$ is outside $\mathbb B(\lambda;\epsilon)$ [resp. $\mathcal{D}_{\lambda,\epsilon}^{i}$], the infimum of $\Psi_i$ on $\overline {\mathbb B(\lambda;\epsilon)}$ [resp. $\mathcal{D}_{\lambda,\epsilon}^{i}$] can only be achieved on the boundary $\partial \mathbb B(\lambda;\epsilon)$ [resp. $\partial \mathcal D^i_{\lambda,\epsilon}$]. Since $\overline{\mathbb B(\lambda;\epsilon)}\subset\mathcal{D}_{\lambda,\epsilon}^{i}$, 
$$
\inf_{x\in\partial \mathcal{D}_{\lambda,\epsilon}^{i}}\Psi_i(x)=\inf_{x\in \mathcal{D}_{\lambda,\epsilon}^i}\Psi_i(x)\le \inf_{x\in \overline{\mathbb{B}(\lambda;\epsilon)}}\Psi_i(x)=\inf_{x\in \partial\mathbb{B}(\lambda;\epsilon)}\Psi_i(x) \ .
$$
 On the other side, by definition, for any point $x$ in $\mathcal{D}_{\lambda,\epsilon}^i$, there exists $x'$ in $\overline{\mathbb B(\lambda;\epsilon)}$ such that $x=\varphi_t^{i,+}(x')$ for some $t$. Since $\Psi_i$ increases along the flow $\varphi^{i,+}$,  
$$
\Psi_i(x)=\Psi_i(\varphi_t^{i,+}(x'))\ge \Psi_i(x')\ge \inf_{z\in \overline{\mathbb{B}(\lambda;\epsilon)}}\Psi_i(z).
$$
 
$\bullet$ For $(b)$: In the case where $\lambda$ is in a close neighborhood of one of the two wells, say~$||\lambda-\lambda_1||=\tilde \varepsilon$ for some $0<\tilde \varepsilon<\epsilon$, then $\mathbb B(\lambda;\epsilon)$ is stable by $-\nabla \Psi_1$ - by convexity of~$\Psi_1$~- and~$\mathcal{D}_{\lambda,\epsilon}^1=\bRb^d$. Since~$\epsilon<\epsilon_0$, $\lambda_2$ is then necessarily located outside $\mathbb B(\lambda;\epsilon)$. Since~$\varphi^{2,+}_t(\lambda_2)=\lambda_2$ for all $t\ge 0$, necessarily $\lambda_2\notin \mathcal D^2_{\lambda,\epsilon}$ and $\mathbb R^d\setminus\mathcal D^2_{\lambda,\epsilon}$ is stable by $-\nabla\Psi_2$. In this case, the related exit-cost is given by
\[
\inf_{y\in\partial \mathcal D^2_{\lambda,\epsilon}}\big(\Psi_2(y)-\Psi_2(\lambda_2)\big)=\inf_{y\in\partial \mathbb B(\lambda;\epsilon)}\big(\Psi_2(y)-\Psi_2(\lambda_2)\big)\,.
\]
The analog can be drawn in the case $||\lambda-\lambda_2||=\tilde \varepsilon$ with the resulting exit-cost:
\[
\inf_{x\in\partial \mathcal D^1_{\lambda,\epsilon}}\big(\Psi_1(x)-\Psi_1(\lambda_1)\big)=\inf_{x\in\partial \mathbb B(\lambda;\epsilon)}\big(\Psi_1(x)-\Psi_1(\lambda_1)\big)\,.
\]
The remaining case ``$(c)$:  $\lambda$ is exactly at a distance $\epsilon$ of $\lambda_1$ or $\lambda_2$'' (that is: one of the wells is located at the boundary of $\mathbb B(\lambda;\epsilon)$)  
is the only situation where the applicability of Theorem~\ref{thm:KramersDZ} of $\mathbb R^d\setminus \mathcal D^i_{\lambda,\epsilon}$ can not be simply identified. This difficulty can be removed by slightly rescaling $\mathcal D^1_{\lambda,\epsilon}\times \mathcal D^2_{\lambda,\epsilon}$ into 
\begin{equation*}
\mathcal O_{\lambda,\epsilon,\rho}:=
\left\{
\begin{aligned}
&\mathcal D^1_{\lambda,\rho\epsilon}\times \mathcal D^2_{\lambda,\epsilon}\,\,\text{if}\,\gaga \lambda-\lambda_{1}\drdr =\epsilon,\\
&\mathcal D^1_{\lambda,\epsilon}\times \mathcal D^2_{\lambda,\rho\epsilon}\,\,\text{if}\,\gaga \lambda-\lambda_{2}\drdr =\epsilon,\\
&\mathcal D^1_{\lambda,\epsilon}\times \mathcal D^2_{\lambda,\epsilon}\,\,\text{otherwise}\,,
\end{aligned}
\right.
\end{equation*}
for $\rho$ arbitrarily chosen in the interval $(0,1)$. Rescaling $\epsilon$ to $\rho \epsilon$ whenever $||\lambda-\lambda_1||=\epsilon$ or~$||\lambda-\lambda_2||=\epsilon$ ensures $\mathcal O_{\lambda,\epsilon,\rho}$ satisfies to the situation $(a)$.

Distinguishing the cases $||\lambda-\lambda_i||=\epsilon$, for $i=1,2$, and according to the discussion above, the set $\mathcal O_{\lambda,\epsilon,\rho}$ is stable by $(-\nabla \Psi_1,-\nabla \Psi_2)$. The related exit-cost
\[
\widehat \h^\rho_\epsilon(\lambda):=\inf_{(x,y)\in\partial \mathcal O_{\lambda,\epsilon,\rho}}\big(\Psi_1(x)+\Psi_2(y)-\Psi_1(\lambda_1)-\Psi_2(\lambda_2)\big)
\]
is equivalently given by
\begin{equation}\label{cost1}
\widehat \h^\rho_\epsilon(\lambda)=\left\{
\begin{aligned}
&\inf_{x\in\partial\mathbb B(\lambda;\rho\epsilon)} \left(\Psi_1(x)-\Psi_1(\lambda_1)\right)+
\inf_{y\in\partial \mathbb B(\lambda;\epsilon)} \left(\Psi_2(y)-\Psi_2(\lambda_2)\right)\,\text{if}\,\gaga \lambda-\lambda_1\drdr=\epsilon,\\
&\inf_{x\in\partial\mathbb B(\lambda;\epsilon)} \left(\Psi_1(x)-\Psi_1(\lambda_1)\right)+\inf_{y\in\partial \mathbb B(\lambda;\rho\epsilon)} \left(\Psi_2(y)-\Psi_2(\lambda_2)\right)
\,\text{if}\,\gaga \lambda-\lambda_2\drdr=\epsilon,\\
&\inf_{x\in\mathcal D^{1}_{\lambda,\epsilon}} \left(\Psi_1(x)-\Psi_1(\lambda_1)\right)+\inf_{
y\in\mathcal D^2_{\lambda,\epsilon}} \left(\Psi_2(y)-\Psi_2(\lambda_2)\right)\,\text{otherwise}\,.
\end{aligned}
\right.
\end{equation}
Further, whenever $\lambda$ is located outside the boundaries $\partial \mathbb B(\lambda_1;\epsilon)$ and  $\partial \mathbb B(\lambda_2;\epsilon)$, 
\begin{align*}
&\inf_{x\in\mathcal D^{1}_{\lambda,\epsilon}} \left(\Psi_1(x)-\Psi_1(\lambda_1)\right)+\inf_{
y\in\mathcal D^2_{\lambda,\epsilon}} \left(\Psi_2(y)-\Psi_2(\lambda_2)\right)
\\
&=\left\{
\begin{aligned}
&\inf_{x\in\partial\mathbb B(\lambda;\epsilon)} \left(\Psi_1(x)-\Psi_1(\lambda_1)\right)+
\inf_{y\in\partial \mathbb B(\lambda;\epsilon)} \left(\Psi_2(y)-\Psi_2(\lambda_2)\right)\,\text{if}\,\min_i\gaga \lambda-\lambda_i\drdr>\epsilon,\\
&\inf_{y\in\partial \mathbb B(\lambda;\epsilon)} \left(\Psi_2(y)-\Psi_2(\lambda_2)\right)
\,\text{if}\,\gaga \lambda-\lambda_1\drdr<\epsilon,\\
&\inf_{x\in \partial \mathbb B(\lambda;\epsilon)} \left(\Psi_1(x)-\Psi_1(\lambda_1)\right) \,\text{if}\,\gaga \lambda-\lambda_2\drdr<\epsilon\,.
\end{aligned}
\right.
\end{align*}

Applying Theorem~\ref{thm:KramersDZ}, we derive the Kramers' type law for the first exit-time 
\begin{equation*}
\widehat{\tau}^\rho_{\lambda,\epsilon}(\sigma):=\inf\left\{t\ge 0\,:\,(x_t^\sigma,y_t^\sigma)\notin (\mathbb R^d\times \mathbb R^d)\setminus \mathcal O_{\lambda,\epsilon,\rho}\right\}\,.
\end{equation*}
\begin{lem}\label{negan}
For any $\lambda$ in $\bRb^d$, $0<\epsilon<\epsilon_0$, $0<\rho<1$ and for any $\delta>0$,

\begin{equation}
\label{gouverneur1}
\lim_{\sigma\to0}\PP\left\{\exp\left[\frac{2}{\sigma^2}\left(\widehat \h^\rho_\epsilon(\lambda)-\delta\right)\right]<
\widehat{\tau}^\rho_{\lambda,\epsilon}(\sigma)<\exp\left[\frac{2}{\sigma^2}\left(\widehat \h^\rho_\epsilon(\lambda)+\delta\right)\right]\right\}=1\,.
\end{equation}
Moreover, we have:    
\begin{equation}
\label{gouverneur2}
\lim_{\sigma\to0}\PP\left\{{\rm dist}\left((x^\sigma_{\widehat{\tau}^\rho_{\lambda,\epsilon}(\sigma)},y^\sigma_{\widehat{\tau}^\rho_{\lambda,\epsilon}(\sigma)}),\mathbb{B}(\lambda;\epsilon)\times\mathbb{B}(\lambda;\epsilon)\right)\leq\delta\right\}=1\,,
\end{equation}
for ${\rm dist}((x,y),\mathbb{B}(\lambda;\eta)\times\mathbb{B}(\lambda;\eta))$
standing for the distance from $(x,y)$ to $\mathbb{B}(\lambda;\eta)\times\mathbb{B}(\lambda;\eta)$.
\end{lem}
 \begin{proof} The asymptotic \eqref{gouverneur1} is a direct consequence of Theorem \ref{thm:KramersDZ}-\eqref{GenericKramers}. The estimate~\eqref{gouverneur2} characterizing the persistence of the first exit-location of $(x^\sigma,y^\sigma)$ on $\mathbb B(\lambda;\epsilon)\times \mathbb B(\lambda;\epsilon)$ follows from Theorem \ref{thm:KramersDZ}-$(2)$. Precisely, as $\sigma\downarrow 0$,  $(x^\sigma_{\widehat{\tau}^\rho_{\lambda,\epsilon}(\sigma)},y^\sigma_{\widehat{\tau}^\rho_{\lambda,\epsilon}(\sigma)})$ concentrates on the points on the boundary $\partial \mathcal O_{\lambda,\epsilon,\rho}$ where the potential 
  \[
 (x,y)\mapsto \Psi_1(x)-\Psi_1(\lambda_1)+\Psi_2(y)-\Psi_2(\lambda_2)
 \]
  is minimal. In view of \eqref{cost1}, these minimizers are located on $\partial \mathbb B(\lambda;\epsilon)$ or  $\partial \mathbb B(\lambda;\rho\epsilon)$. And so the exit-location has to persist on $\mathbb{B}(\lambda;\epsilon)\times\mathbb{B}(\lambda;\epsilon)$.
\end{proof}

From Lemma \ref{negan}, we gradually derive a Kramers' type law for $\tau_{\lambda,\epsilon}(\sigma)$ through the two following lemmas. 
\begin{lem}
\label{dale}
Define
\begin{equation*}
\tau^\rho_{\lambda,\epsilon}(\sigma)=\left\{
\begin{aligned}
&\inf\left\{t\ge 0\,:\,(x_t^\sigma,y_t^\sigma)\in \mathbb B(\lambda;\rho\epsilon)\times \mathbb B(\lambda;\epsilon)\right\}\,\,\text{if}\,\gaga \lambda-\lambda_{1}\drdr =\epsilon,\\
&\inf\left\{t\ge 0\,:\,(x_t^\sigma,y_t^\sigma)\in \mathbb B(\lambda;\epsilon)\times \mathbb B(\lambda;\rho\epsilon)\right\}\,\,\text{if}\,\gaga \lambda-\lambda_{2}\drdr =\epsilon,\\
&\inf\left\{t\ge 0\,:\,(x_t^\sigma,y_t^\sigma)\in \mathbb B(\lambda;\epsilon)\times \mathbb B(\lambda;\epsilon)\right\}\,\,\text{otherwise}\,.
\end{aligned}
\right.
\end{equation*}
Then, for any $\lambda\in\bRb^d$,  $0<\epsilon<\epsilon_0$, $0<\rho<1$ and for any $\delta>0$:

\begin{equation}
\label{merle}
\lim_{\sigma\to0}\PP\left\{\exp\left[\frac{2}{\sigma^2}\left(\widehat\h^\rho_\epsilon(\lambda)-\delta\right)\right]<
\tau^\rho_{\lambda,\epsilon}(\sigma)<\exp\left[\frac{2}{\sigma^2}\left(\widehat\h^\rho_\epsilon(\lambda)+\delta\right)\right]\right\}=1\,.
\end{equation}
Moreover \eqref{gouverneur2} still holds true with $(x_{\tau^\rho_{\lambda,\epsilon}(\sigma)},y_{\tau^\rho_{\lambda,\epsilon}(\sigma)})$ in place of $(x_{\widehat\tau^\rho_{\lambda,\epsilon}(\sigma)},y_{\widehat\tau^\rho_{\lambda,\epsilon}(\sigma)})$.
\end{lem}

\begin{proof}
Since ${\mathbb B}(\lambda;\epsilon)$ and ${\mathbb B}(\lambda;\rho\epsilon)$ are contained in each $\mathcal D_{\lambda,\epsilon}^i$, necessarily the inequality~$\tau^\rho_{\lambda,\epsilon}(\sigma)\ge \widehat{\tau}^\rho_{\lambda,\epsilon}(\sigma)$ holds almost surely. Since \eqref{gouverneur1} ensures that
\[
\lim_{\sigma\to0}\PP\left\{\widehat{\tau}^\rho_{\lambda,\epsilon}(\sigma)\le \exp\left[\frac{2}{\sigma^2}\left(\widehat \h^\rho_\epsilon(\lambda)-\delta\right)\right]\right\}=0\,,
\]
the lower tail in \eqref{merle} follows. To establish the upper-tail 
\[
\lim_{\sigma\to0}\PP\left\{\tau^\rho_{\lambda,\epsilon}(\sigma)<\exp\left[\frac{2}{\sigma^2}\left(\widehat \h^\rho_\epsilon(\lambda)+\delta\right)\right]\right\}=1\,,
\]
fix $\delta>0$, let $\xi>0$ be smaller than $\rho$ and use the inequality:
\begin{align*}
&\mathbb P\left\{ \tau^\rho_{\lambda,\epsilon}(\sigma)\ge \exp\left[\frac{2}{\sigma^2}\left(\widehat\h^\rho_\epsilon(\lambda)+\delta\right)\right]\right\}\\
&\le \mathbb P\left\{\widehat{\tau}^\xi_{\lambda,\epsilon}(\sigma)\ge  \exp\left[\frac{2}{\sigma^2}\left(\widehat\h^\rho_\epsilon(\lambda)+\delta\right)\right] \right\}\\
&+\mathbb P\left\{ \tau^\rho_{\lambda,\epsilon}(\sigma)\ge \exp\left[\frac{2}{\sigma^2}\left(\widehat\h^\rho_\epsilon(\lambda)+\delta\right)\right],\,\widehat{\tau}^\xi_{\lambda,\epsilon}(\sigma)< \tau^\rho_{\lambda,\epsilon}(\sigma)\right\}\\
&\le \mathbb P\left\{\widehat{\tau}^\xi_{\lambda,\epsilon}(\sigma)\ge  \exp\left[\frac{2}{\sigma^2}\left(\widehat\h^\rho_\epsilon(\lambda)+\delta\right)\right] \right\}
+\mathbb P\left\{\widehat{\tau}^\xi_{\lambda,\epsilon}(\sigma)< \tau^\rho_{\lambda,\epsilon}(\sigma)\right\}.
\end{align*}
Observing that $\eta\mapsto \widehat\h^\eta_\epsilon(\lambda)$ is continuous, we can choose $\xi$ close enough to $\rho$ so that~$\widehat\h^\rho_\epsilon(\lambda)>\widehat\h^\xi_\epsilon(\lambda)-\delta/2$. This way, the event $\left\{\widehat{\tau}^\xi_{\lambda,\epsilon}(\sigma)\ge  \exp\left[\frac{2}{\sigma^2}\left(\widehat\h^\rho_\epsilon(\lambda)+\delta\right)\right]\right\}$ is included into the event $\left\{\widehat{\tau}^\xi_{\lambda,\epsilon}(\sigma)\ge  \exp\left[\frac{2}{\sigma^2}\left(\widehat\h^\xi_\epsilon(\lambda)+\frac{\delta}{2}\right)\right]\right\}$ and the upper-tail estimate for $\widehat{\tau}^\xi_{\lambda,\epsilon}(\sigma)$ in~\eqref{gouverneur1} ensures that $\mathbb P\left\{\widehat{\tau}^\xi_{\lambda,\epsilon}(\sigma)\ge  \exp\left[\frac{2}{\sigma^2}\left(\widehat\h^\rho_\epsilon(\lambda)+\delta\right)\right] \right\}$ vanishes as $\sigma$ tends to~$0$. For the remaining component, the event $\{\widehat{\tau}^\xi_{\lambda,\epsilon}(\sigma)< \tau^\rho_{\lambda,\epsilon}(\sigma)\}$ implies that the vector $\big(x^\sigma_{\widehat{\tau}^\xi_{\lambda,\epsilon}(\sigma)},y^\sigma_{\widehat{\tau}^\xi_{\lambda,\epsilon}(\sigma)}\big)$ does not belong to $\overline{\mathbb B(\lambda;\epsilon)}\times\overline{\mathbb B(\lambda;\epsilon)}$. Recalling~\eqref{gouverneur2} from Lemma \ref{negan}, this event becomes negligible as $\sigma\downarrow 0$ and so
$\mathbb P\left\{\widehat{\tau}^\xi_{\lambda,\epsilon}(\sigma)< \tau^\rho_{\lambda,\epsilon}(\sigma)\right\}$ vanishes as $\sigma$ tends to $0$.
 
  The persistence of the first collision-location $(x_{\tau^\rho_{\lambda,\epsilon}(\sigma)},y_{\tau^\rho_{\lambda,\epsilon}(\sigma)})$ is a straightforward consequence of the very definition of $\tau_{\lambda,\epsilon}^\rho(\sigma)$.
\end{proof}

 \begin{lem}
\label{dale-bis}
Let $\tau_{\lambda,\epsilon}(\sigma)$ be defined as in \eqref{andrea}.
For any $\lambda\in\bRb^d$ and $0<\epsilon<\epsilon_0$, it holds: for any $\delta>0$,
\begin{equation}\label{shane}
\lim_{\sigma\to0}\PP\left\{\exp\left[\frac{2}{\sigma^2}\left(\widehat\h_\epsilon(\lambda)-\delta\right)\right]<
\tau_{\lambda,\epsilon}(\sigma)<\exp\left[\frac{2}{\sigma^2}\left(\widehat\h_\epsilon(\lambda)+\delta\right)\right]\right\}=1\,,
\end{equation}
for $\widehat\h_\epsilon$ given by

\begin{align}\label{gouverneur3}
\widehat\h_\epsilon(\lambda)&:=\lim_{\rho\rightarrow 1}\widehat\h^\rho_{\epsilon}(\lambda)\nonumber\\
&=\left\{
\begin{aligned}
&\inf_{y\in\partial \mathbb B(\lambda;\epsilon)}\big(\Psi_2(y)-\Psi_2(\lambda_2)\big)\,\mbox{if}\,\,||\lambda-\lambda_1||< \epsilon\,,\\
&\inf_{x\in\partial \mathbb B(\lambda;\epsilon)}\big(\Psi_1(x)-\Psi_1(\lambda_1)\big)\,\mbox{if}\,\,||\lambda-\lambda_2||< \epsilon\,\\
&\inf_{x\in\partial \mathbb B(\lambda;\epsilon)}\big(\Psi_1(x)-\Psi_1(\lambda_1)\big)+\inf_{y\in\partial \mathbb B(\lambda;\epsilon)}\big(\Psi_2(y)-\Psi_2(\lambda_2)\big)\,\mbox{if}\,\,\min_{i=1,2}||\lambda-\lambda_i||\ge \epsilon\,.	
\end{aligned}
\right. 
\end{align}
Additionally, 
\begin{equation}
\label{gouverneur4}
\lim_{\sigma\to0}\PP\left\{{\rm dist}\left((x^\sigma_{\tau_{\lambda,\epsilon}(\sigma)},y^\sigma_{\tau_{\lambda,\epsilon}(\sigma)}),\mathbb{B}(\lambda;\epsilon)\times\mathbb{B}(\lambda;\epsilon)\right)\leq\delta\right\}=1\,.
\end{equation}
\end{lem}
\begin{proof}  
The estimate \eqref{gouverneur4} is again straightforward.

By definition, $\tau_{\lambda,\epsilon}(\sigma)\le \tau^\rho_{\lambda,\epsilon}(\sigma)$ almost surely, and so, for any $\delta'>0$,
\[
\lim_{\sigma\rightarrow 0}\mathbb P\left\{\tau_{\lambda,\epsilon}(\sigma)<\exp\left[\frac 2{\sigma^2}(h^\rho_\epsilon(\lambda)+\delta')\right]\right\} 
=1\,.
\] 
Also as $\lim_{\rho\rightarrow 1}\widehat\h^\rho_\epsilon(\lambda)=\widehat\h_\epsilon(\lambda)$, taking $\delta>0$ arbitrary, and choosing $\rho,\delta'$ small enough so that $\widehat\h^\rho_\epsilon(\lambda)+\delta'\le \widehat\h_\epsilon(\lambda)+\delta$ yields the upper-tail:
\[
\lim_{\sigma\rightarrow 0}\mathbb P\left\{\tau_{\lambda,\epsilon}(\sigma)<\exp\left[\frac 2{\sigma^2}(\widehat\h_\epsilon(\lambda)+\delta)\right]\right\}=1\,.
\] 
For the lower-tail: 
\[
\lim_{\sigma\rightarrow 0}\mathbb P\left\{\tau_{\lambda,\epsilon}(\sigma)>\exp\left[\frac 2{\sigma^2}(\widehat\h_\epsilon(\lambda)-\delta)\right]\right\}=1\,,
\] 
let us consider the situation $||\lambda-\lambda_1||=\epsilon$ which implies that $||\lambda-\lambda_2||>\epsilon$, $\mathbb R^d\setminus\mathcal D^2_{\lambda,\epsilon}$ is stable by $-\nabla\Psi_2$ and $\widehat\h_\epsilon(\lambda)$ reduces to $\inf_{y\in\partial\mathbb B(\lambda;\epsilon)}\big(\Psi_2(y)-\Psi_2(\lambda_2)\big)$. 
Observing that $\tau_{\lambda,\epsilon}(\sigma)$ is greater or equal to $\tilde \tau_{\lambda,\epsilon}(\sigma):=\inf\{t\ge 0\,:\,y_t^\sigma\in\mathbb B(\lambda;\epsilon)\}$, the lower-tail estimate follows from the inequality 
\begin{align*}
&\mathbb P\left\{\tau_{\lambda,\epsilon}(\sigma)>\exp\left[\frac 2{\sigma^2}(\widehat\h_\epsilon(\lambda)-\delta)\right] \right\}\\
&\ge 
\mathbb P\left\{\tilde \tau_{\lambda,\epsilon}(\sigma)>\exp\left[\frac 2{\sigma^2}\bigg(\inf_{y\in\partial\mathbb B(\lambda;\epsilon)}(\Psi_2(y)-\Psi_2(\lambda_2))-\delta\bigg) \right]\right\}
\end{align*}
and, by applying Theorem \ref{thm:KramersDZ} to $\tilde \tau_{\lambda,\epsilon}(\sigma)$. 

\noindent
Following the same reasoning in the case $||\lambda-\lambda_2||=\epsilon$, the claim follows. Finally, whenever $||\lambda-\lambda_1||\neq\epsilon$ and $||\lambda-\lambda_2||\neq\epsilon$, $\tau_{\lambda,\epsilon}(\sigma)$ simply reduces to $\widehat{\tau}^\rho_{\lambda,\epsilon}(\sigma)$.
\end{proof}

\subsection{Asymptotic estimates for $c_\epsilon(\sigma)$}\label{subsec:Linear-collisionB}
Following Lemma \ref{dale-bis}, the characteristics for the Kramers' type law of $c_\epsilon(\sigma)$ can be drawn heuristically: assuming that \eqref{shane} and \eqref{gouverneur4} are stable by minimization over the intermediate points $\lambda$, the exit-cost governing the asymptotic~$\sigma\downarrow 0$ would be given by $\inf_\lambda\widehat\h_\epsilon(\lambda)$ and the exit-location $(x^\sigma_{c_\epsilon(\sigma)},y^\sigma_{c_\epsilon(\sigma)})$ would concentrate on the domains $\mathbb B(\lambda_\epsilon;\epsilon)\times \mathbb B(\lambda_\epsilon;\epsilon)$ where $\lambda_\epsilon$ belongs to the set of minimizers of $\widehat\h_{\epsilon}$. This set possibly contains multiple elements and, in view of \eqref{gouverneur3} can be split into three main subsets: the family of minimizers belonging to $\mathbb B(\lambda_1;\epsilon)$ or to $\mathbb B(\lambda_2;\epsilon)$ 
  - the respective minima of $\widehat\h_{\epsilon}$ being given by
 \[
 m_{1,\epsilon}:=\inf_{\lambda \in \mathbb B(\lambda_1;\epsilon)}\inf_{y\in\partial \mathbb B(\lambda;\epsilon)}\big(\Psi_2(y)-\Psi_2(\lambda_2)\big),\:\:\: m_{2,\epsilon}:=\inf_{\lambda \in \mathbb B(\lambda_2;\epsilon)}\inf_{x\in\partial \mathbb B(\lambda;\epsilon)}\big(\Psi_1(x)-\Psi_1(\lambda_1)\big),
 \]
 and the family of minimizers belonging to $\mathbb R^d\setminus \big(\mathbb B(\lambda_1;\epsilon)\cup\mathbb B(\lambda_2;\epsilon)\big) $ -  the corresponding minima being given by $m_\epsilon:=\inf_{\{\lambda\::\:\inf_i||\lambda-\lambda_i||\ge \epsilon\}} h_\epsilon(\lambda)$ for $h_\epsilon$ being itself given by
 \[
 \h_\epsilon(\lambda):=\inf_{x\in\partial \mathbb B(\lambda;\epsilon)}\big(\Psi_1(x)-\Psi_1(\lambda_1)\big)+ \inf_{y\in\partial \mathbb B(\lambda;\epsilon)}\big(\Psi_2(y)-\Psi_2(\lambda_2)\big)\,.
 \]
Provided that $\epsilon$ is small enough, the latter will predominate over the two others and will specifically drive the persistence of $(x^\sigma_{c_\epsilon(\sigma)},y^\sigma_{c_\epsilon(\sigma)})$. Indeed, as $\epsilon\downarrow 0$, $\widehat \h_\epsilon$ converges, on any compact, to the function 
$$
\h_0(\lambda):=\left(\Psi_1(\lambda)-\Psi_1(\lambda_1)\right)+\left(\Psi_2(\lambda)-\Psi_2(\lambda_2)\right)\,,
$$
The potentials $\Psi_1$ and $\Psi_2$ being uniformly  
convex, $\h_0$ is uniformly convex as well.
As such, $\lim_{\epsilon \rightarrow 0}\inf_\lambda \widehat\h_{\epsilon}(\lambda)=\inf_\lambda \h_{0}(\lambda)$. Additionally, the set of minimizers of $\widehat\h_\epsilon$ converges to a unique point, $\text{argmin}\,\h_0$, which as $\Psi_1$ and $\Psi_2$ are of class $\mathcal C^2$, is explicitly given by~$(\nabla \Psi_1+\nabla \Psi_2)^{-1}(0)$. Observing further that $\inf_\lambda \h_0$ is strictly smaller than $\lim_{\epsilon\rightarrow 0}m_{1,\epsilon}=\Psi_1(\lambda_2)-\Psi_1(\lambda_1)$ and than $\lim_{\epsilon\rightarrow 0}m_{2,\epsilon}=\Psi_2(\lambda_1)-\Psi_2(\lambda_2)$, we can introduce the (positive) threshold
\begin{equation*}
 \epsilon_c:=\inf\left\{\epsilon\,\in\,(0,\epsilon_0)\,:\,
 m_{i,\epsilon}=\inf_\lambda \h_\epsilon(\lambda),\,\,i\in\{1,2\}\right\}
 \end{equation*} 
 which corresponds to the smallest radius for which the exit-cost $\inf_\lambda\widehat\h_\epsilon(\lambda)$ reduces into $\inf_{\lambda}\h_\epsilon(\lambda)$. Let us also notice for any $\epsilon<\epsilon_c$, $\inf_{\lambda} h_\epsilon(\lambda)=\inf_{\{\lambda\::\:\inf_i||\lambda-\lambda_i||\ge \epsilon\}} h_\epsilon(\lambda)$.
 
 Following this preliminary discussion, we can state the two main results of this section.   
\begin{prop}
\label{lacollision} For $\epsilon<\epsilon_c$ and for $\underline{\h}_{\epsilon}=\inf_{\lambda \in \bRb^d}\h_{\epsilon}(\lambda)$, for any $\delta>0$, we have:

\begin{equation}
\label{tdog}
\lim_{\sigma\to0}\PP\left\{\exp\left[\frac{2}{\sigma^2}\left(\underline{\h}_{\epsilon} -\delta\right)\right]<
c_\epsilon(\sigma)<\exp\left[\frac{2}{\sigma^2}\left(\underline{\h}_{\epsilon}+\delta\right)\right]\right\}=1\,.
\end{equation}

In addition,  for $\mathcal M_\epsilon$ the set of minimizers of $\lambda\mapsto \h_{\epsilon}(\lambda)$, it holds:
\begin{equation}
\label{tdog2}
\lim_{\sigma\to0}\PP\left\{\inf_{\lambda_\epsilon\in\mathcal M_\epsilon}\max\bigg({\rm dist}\big(x^\sigma_{c_\epsilon(\sigma)},\mathbb B(\lambda_\epsilon;\epsilon)\big),{\rm dist}\big(y^\sigma_{c_\epsilon(\sigma)},\mathbb B(\lambda_\epsilon;\epsilon) \big)\bigg)\ge \delta\right\}=0\,.
\end{equation}
 \end{prop}

\begin{proof}
Since $c_\epsilon(\sigma)$ is the minimum of the possible stopping-times~$\tau_{\lambda,\epsilon}(\sigma)$ over all $\lambda$, the upper-tail estimate in \eqref{tdog} follows immediately from the upper-tail~\eqref{shane} in Lemma~\ref{dale-bis} applied to $\tau_{\lambda_\epsilon,\epsilon}(\sigma)$ for $\lambda_\epsilon$ a minimizer of the function~$\h_{\epsilon}$.

To obtain the lower-tail estimate and the estimate  \eqref{tdog2}, it is sufficient to show that the barycenter $z^\sigma$ of $x^\sigma$ and $y^\sigma$, defined by $z^\sigma_t:=(x^\sigma_t+y^\sigma_t)/2$, satisfies the exit-property:
{\color{black}
\begin{equation}\label{ben}
\lim_{\sigma\rightarrow 0}\mathbb P\left\{z^\sigma_{c_\epsilon(\sigma)}\in \overline{\mathcal{M}_{\epsilon,\delta}}\right\}= 1,
\end{equation}
where $\mathcal{M}_{\epsilon,\delta}:=\{z\in\mathbb R^d\,:\,\text{dist}(z,\mathcal{M}_{\epsilon})<\delta\}$ for $\delta$ arbitrary positive. Indeed, since $||x^\sigma_{c_\epsilon(\sigma)}-y^\sigma_{c_\epsilon(\sigma)}||= 2\epsilon$, \eqref{tdog2} and \eqref{ben}} are equivalent. Additionally, since
\begin{align*}
&\mathbb P\left\{c_\epsilon(\sigma)
\le \exp\left[\frac{2}{\sigma^2}(\underline \h_\epsilon-\delta)\right]\right\}\\
&\le \mathbb P\left\{z^\sigma_{c_\epsilon(\sigma)}\notin \overline{\mathcal{M}_{\epsilon,\delta}}\right\}
+\mathbb P\left\{c_\epsilon(\sigma)\le \exp\left[\frac{2}{\sigma^2}(\underline \h_\epsilon-\delta)\right],\,\,z^\sigma_{c_\epsilon(\sigma)}\in \overline{\mathcal{M}_{\epsilon,\delta}}\right\}\,,
\end{align*}
the limit \eqref{ben} reduces the proof of the lower tail in \eqref{tdog} to establish that
\begin{align*}
\lim_{\sigma\rightarrow 0}\mathbb P\left\{c_\epsilon(\sigma)\le \exp\left[\frac{2}{\sigma^2}(\underline \h_\epsilon-\delta)\right],\,\,z^\sigma_{c_\epsilon(\sigma)}\in \overline{\mathcal{M}_{\epsilon,\delta}}\right\}=0\,.
\end{align*}
For $\delta>0$, since $\overline{\mathcal{M}_{\epsilon,\delta}}$ is compact (this property being a consequence of the compactness of $\mathcal M_\epsilon$  which follows from the convexity of $\Psi_1$ and of $\Psi_2$), one can construct a finite covering~$\cup_{l=1}^L\mathbb B(\lambda^l;r)\supset \overline{\mathcal{M}_{\epsilon,\delta}}$ - where $r$, aimed to be small, will be chosen later on - and for which the event
\[
\left\{c_\epsilon(\sigma)\le \exp\left[\frac{2}{\sigma^2}(\underline \h_\epsilon-\delta)\right],\,z^\sigma_{c_\epsilon(\sigma)}\in \overline{\mathcal{M}_{\epsilon,\delta}}\right\}
\]
is embedded in the union
\[
\bigcup_{l=1}^L\left\{c_\epsilon(\sigma)\le \exp\left[\frac{2}{\sigma^2}(\underline \h_\epsilon-\delta)\right],\,z^\sigma_{c_\epsilon(\sigma)}\in \overline{\mathbb B(\lambda^l;r)}\right\}.
\]
For any $l$,  the event $\{z^\sigma_{c_\epsilon(\sigma)}\in\overline{\mathbb B(\lambda^l;r)}\}$ implies that the events $\left\{|| x^\sigma_{c_\epsilon(\sigma)}-\lambda^l||\le \epsilon+r\right\}$ and~$\left\{|| y^\sigma_{c_\epsilon(\sigma)}-\lambda^l ||\le \epsilon+r\right\}$ occur simultaneously, and so does $\{c_\epsilon(\sigma)\ge \tau_{\lambda^l,\epsilon+r}(\sigma)\}$. Choosing~$r$ small enough so that $\h_{\epsilon+r}(\lambda^l)-\delta'\ge \underline \h_\epsilon-\delta$ for $\delta'>0$, Lemma~\ref{dale-bis} yields, for any $l$,
\begin{align*}
&\lim_{\sigma\rightarrow 0}\mathbb P\left\{\tau_{\lambda^l,\epsilon+r}(\sigma)\le \exp\left[\frac{2}{\sigma^2}\bigg(\underline \h_\epsilon-\delta\bigg)\right]\right\}\\
&\le \lim_{\sigma\rightarrow 0}\mathbb P\left\{\tau_{\lambda^l,\epsilon+r}(\sigma)\le \exp\left[\frac{2}{\sigma^2}\bigg(\h_{\epsilon+r}(\lambda^l)-\delta'\bigg)\right]\right\}=0\,.
\end{align*}
Therefore 
\begin{align*}
\sum_{l=1}^L\mathbb P\left\{c_\epsilon(\sigma)\le \exp\left[\frac{2}{\sigma^2}(\underline \h_\epsilon-\delta)\right],\,z^\sigma_{c_\epsilon(\sigma)}\in \overline{\mathbb B(\underline \lambda^l;r)}\right\}
\end{align*}
vanishes as $\sigma$ tends to $0$, yielding
\[
\lim_{\sigma\rightarrow 0}\mathbb P\left\{c_\epsilon(\sigma)\le \exp\left[\frac{2}{\sigma^2}(\underline \h_\epsilon-\delta)\right]\right\}=0\,.
\]

Let us now establish \eqref{ben}. To this aim, for $\xi>0$, define the level set
\begin{equation*}
S_{\epsilon,\xi}=\left\{\lambda\in\mathbb R^d\,:\,\Psi_1(\lambda)-\Psi_1(\lambda_1)\ge \underline\h_\epsilon+3\xi\right\}\,.
\end{equation*}
Since $\epsilon<\epsilon_c$ ensures that $\lambda_1\notin \mathcal M_\epsilon$, $S_{\epsilon,\xi}$ is necessarily non-empty. Define next, $M_{\epsilon,\xi}$ a minimal radius, strictly larger than $\epsilon$ and such that $\mathbb R^d\setminus\mathbb B(\lambda_1;M_{\epsilon,\xi}-\epsilon)$ lies in $S_{\epsilon,\xi}$. Equivalently, for any $\lambda\in\mathbb R^d$ such that $||\lambda-\lambda_1||>M_{\epsilon,\xi}-\epsilon$, the difference $\Psi_1(\lambda)-\Psi_1(\lambda_1)$ is larger than~$\underline\h_\epsilon+3\xi$. From this, one can derive the bound:
\begin{align*}
&\mathbb P\left\{z^\sigma_{c_\epsilon(\sigma)}\notin \overline{\mathcal{M}_{\epsilon,\delta}}\right\}\\
&\le \mathbb P\left\{z^\sigma_{c_\epsilon(\sigma)}\notin \overline{\mathbb B(\lambda_1;M_{\epsilon,\xi})}\right\}
+\mathbb P\left\{z^\sigma_{c_\epsilon(\sigma)}\in \overline{\mathbb B(\lambda_1;M_{\epsilon,\xi})}\setminus \overline{\mathcal{M}_{\epsilon,\delta}}\right\}=:I_1(\sigma)+I_2(\sigma)\,,
\end{align*}
and check that $\lim_{\sigma\rightarrow 0}I_i(\sigma)=0$ for $i=1,2$.

For the limit of $I_1(\sigma)$: using again the fact that $||x^\sigma_{c_\epsilon(\sigma)}-y^\sigma_{c_\epsilon(\sigma)}||=2\epsilon$, the distance between~$x^\sigma_{c_\epsilon(\sigma)}$ and $\lambda_1$ is larger than $|| z^\sigma_{c_\epsilon(\sigma)}-\lambda_1||-\epsilon$. Introducing  
the first time $x^\sigma$ exits the ball~$\mathbb B(\lambda_1;M_{\epsilon,\xi}-\epsilon)$:
\[
\tilde \tau_{M_{\epsilon,\xi}-\epsilon}(\sigma):=\inf\left\{t\ge 0\,:\,\gaga x^\sigma_t-\lambda_1\drdr \ge M_{\epsilon,\xi}-\epsilon\right\}\,,
\]
we have 
\begin{align*}
\mathbb P\left\{z^{\sigma}_{c_\epsilon(\sigma)}\notin \overline{\mathbb{B}(\lambda_1;M_{\epsilon,\xi})}\right\}\le \mathbb P\left\{x^{\sigma}_{c_\epsilon(\sigma)}\notin \overline{\mathbb{B}(\lambda_1;M_{\epsilon,\xi}-\epsilon)}\right\}\le \mathbb P\left\{\tilde \tau_{M_{\epsilon,\xi}-\epsilon}\le c_\epsilon(\sigma)\right\}\,.
\end{align*}
Since $\mathbb B(\lambda_1;M_{\epsilon,\xi}-\epsilon)$ is stable by $-\nabla \Psi_1$ - (again) by the convexity of $\Psi_1$ -,
Theorem \ref{thm:KramersDZ} applies for $\tilde{\tau}_{M_{\epsilon,\xi}-\epsilon}$ with the exit-cost
\begin{equation*}
\tilde h_\epsilon:=\inf_{x\in \partial \mathbb{B}(\lambda_1;M_{\epsilon,\xi}-\epsilon)}\{\Psi_1(x)-\Psi_1(\lambda_1)\}\geq\underline\h_\epsilon+3\xi\,.
\end{equation*}
In addition, since, for any $\delta'>0$,
\begin{align*}
\mathbb P\left\{\tilde \tau_{M_{\epsilon,\xi}-\epsilon}(\sigma)\le c_\epsilon(\sigma)\right\}
&\le \mathbb P\left\{\tilde{\tau}_{M_{\epsilon,\xi}-\epsilon}(\sigma)\le \exp\left[\frac{2}{\sigma^2}(\underline \h_\epsilon+\delta')\right]\right\}\\
&+\mathbb P\left\{c_\epsilon(\sigma)> \exp\left[\frac{2}{\sigma^2}(\underline \h_\epsilon+\delta')\right]\right\},
\end{align*}
choosing $\delta'<3\xi$ so that $\underline \h_\epsilon+\delta'\le \tilde \h_\epsilon-\delta''$ for $\delta''>0$ yields
\begin{align*}
\lim_{\sigma\rightarrow 0}\mathbb P\left\{\tilde \tau_{M_{\epsilon,\xi}-\epsilon}\le \exp\left[\frac{2}{\sigma^2}(\underline \h_\epsilon+\delta')\right]\right\}\le \lim_{\sigma\rightarrow 0}\mathbb P\left\{\tilde \tau_{M_{\epsilon,\xi}-\epsilon}\le \exp\left[\frac{2}{\sigma^2}(\tilde \h_\epsilon-\delta'')\right]\right\}=0\,.
\end{align*}
Owing to the upper-tail of $c_\epsilon(\sigma)$ established at the beginning of the proof, we deduce that
\begin{align*}
&\lim_{\sigma\rightarrow 0}\mathbb P\left\{z^{\sigma}_{c_\epsilon(\sigma)}\notin \overline{\mathbb{B}(\lambda_1;M_{\epsilon,\xi})}\right\}\\
&\le \lim_{\sigma\rightarrow 0}\mathbb P\left\{\tilde \tau_{M_{\epsilon,\xi}-\epsilon}\le \exp\left[\frac{2}{\sigma^2}(\underline \h_\epsilon+\delta')\right]\right\}+\lim_{\sigma\rightarrow 0}\mathbb P\left\{c_\epsilon(\sigma)> \exp\left[\frac{2}{\sigma^2}(\underline \h_\epsilon+\delta')\right]\right\}=0\,.
\end{align*} 

For the limit of $I_2(\sigma)$: introducing a new covering $\cup_{l=1}^{\hat L} \mathbb B(\hat \lambda^l;\hat r)\supset \overline{\mathbb B(\lambda_1;M_{\epsilon,\xi})}\setminus \overline{\mathcal{M}_{\epsilon,\delta}}$~-~where $\hat r$ will be again chosen later on - we derive the upper-bound
\begin{align*}
\mathbb P\bigg\{z^\sigma_{c_\epsilon(\sigma)}\in \overline{\mathbb B(\lambda_1;M_{\epsilon,\xi})}\setminus \overline{\mathcal{M}_{\epsilon,\delta}}\bigg\}\le \sum_{l=1}^{\hat L}\mathbb P\bigg\{z^\sigma_{c_\epsilon(\sigma)}\in \mathbb B(\hat \lambda^l;\hat r)\bigg\}\,.
\end{align*}
For any $l$, we also have
\begin{align*}
\mathbb P\bigg\{z^\sigma_{c_\epsilon(\sigma)}\in \mathbb B(\hat \lambda^l;\hat r)\bigg\}
&\le \mathbb P\bigg\{x^\sigma_{c_\epsilon(\sigma)}\in \mathbb B(\hat \lambda^l;\hat r+\epsilon),y^\sigma_{c_\epsilon(\sigma)}\in \mathbb B(\hat \lambda^l;\hat r+\epsilon)\bigg\}\\
&\le \mathbb P\bigg\{\tau_{\hat \lambda^l,\epsilon+\hat r}(\sigma)\le c_\epsilon(\sigma)\bigg\}\,.
\end{align*}
According to Lemma \ref{dale-bis}, for any $\delta>0$,
\begin{align*}
\lim_{\sigma\rightarrow 0}\mathbb P\bigg\{\tau_{\hat \lambda^l,\epsilon+\hat r}(\sigma)\le \exp\left[\frac{2}{\sigma^2}(\h_{\epsilon+\hat r}(\hat \lambda^l)-\delta)\right]\bigg\}=0\,.
\end{align*}
Choosing $\hat r$ and $\delta$ so that, for some $\hat\delta>0$, $\h_{\epsilon+\hat r}(\hat \lambda^l)-\hat \delta\ge  \underline \h_\epsilon+\delta$ then gives
\begin{align*}
&\lim_{\sigma\rightarrow 0}\mathbb P\bigg\{\tau_{\hat \lambda^l,\epsilon+\hat r}(\sigma)\le \exp\left[\frac{2}{\sigma^2}(\underline \h_\epsilon+\delta)\right]\bigg\}\\
&\le
\lim_{\sigma\rightarrow 0} \mathbb P\bigg\{\tau_{\hat \lambda^l,\epsilon+\hat r}(\sigma)\le \exp\left[\frac{2}{\sigma^2}(\h_{\epsilon+\hat r}(\hat \lambda^l)-\hat \delta)
\right]\bigg\}=0.
\end{align*}
Therefore, using the inequality
\begin{align*}
\mathbb P\left\{\tau_{\hat \lambda^l,\epsilon+\hat r}(\sigma)\le c_\epsilon(\sigma)\right\}&\le \mathbb P\left\{\tau_{\hat \lambda^l,\epsilon+\hat r}(\sigma)\le \exp\left[\frac{2}{\sigma^2}(\underline \h_\epsilon+\delta)\right]\right\}\\
&+\mathbb P\left\{c_\epsilon(\sigma)\ge \exp\left[\frac{2}{\sigma^2}(\underline \h_\epsilon+\delta)\right]\right\}\,,
\end{align*}
and again Lemma \ref{dale-bis} and the upper-tail of $c_\epsilon(\sigma)$ yields to
\[
\lim_{\sigma\rightarrow 0}\sum_{l=1}^{\hat{L}}\mathbb P\bigg\{z^\sigma_{c_\epsilon(\sigma)}\in \mathbb B(\hat \lambda^l;\hat r)\bigg\}=0\,.
\] 
This immediately implies that $\lim_{\sigma\rightarrow 0}I_2(\sigma)=0$ and ends the proof of \eqref{ben}.  
\end{proof}
 
\begin{thm}
\label{lacollisionthm} For $\lambda_0$ the minimizer of $\h_0$ and $\underline{\h}_0:=\h_0(\lambda_0)$, for any $\delta>0$, we have:

\begin{equation*}
\lim_{\epsilon\to0}
\lim_{\sigma\to0}\PP\left\{\exp\left[\frac{2}{\sigma^2}\left(\underline{\h}_0
-\delta\right)\right]<c_\epsilon(\sigma)<
\exp\left[\frac{2}{\sigma^2}\left(\underline{\h}_0+\delta\right)\right]\right\}=1\,
\end{equation*}
and
\begin{equation*}
\lim_{\epsilon\to0}\lim_{\sigma\to0}\PP\left\{\max\bigg(\gaga x^\sigma_{c_\epsilon(\sigma)}-\lambda_0\drdr,\gaga y^\sigma_{c_\epsilon(\sigma)}-\lambda_0\drdr\bigg)\le \delta\right\}=1\,.
\end{equation*}

\end{thm}

Theorem \ref{lacollisionthm} provides an analog of Theorem \ref{thm:main1} and is a straightforward consequence of Proposition  \ref{lacollision} and of the limit $\lim_{\epsilon\rightarrow 0}\underline h_\epsilon=\underline h_0$.  
 \begin{rem} Let us briefly comment on the set of minimizers $\mathcal M_\epsilon$ and highlight the possible collision-locations $\lambda_\epsilon\in\mathcal M_\epsilon$ in Proposition \ref{lacollision}. To this aim, observe 
\begin{align*}
\inf_\lambda h_\epsilon(\lambda)&= \inf_{\lambda}\inf_{(x,y)\in\partial \mathbb{B}(0;\epsilon)\times \partial \mathbb{B}(0;\epsilon)} \Big(\Psi_1(\lambda+x)-\Psi_1(\lambda_1)+\Psi_2(\lambda+y)-\Psi_2(\lambda_2)\Big)\\
&=\inf_{(x,y)\in\partial \mathbb{B}(0;\epsilon)\times \partial \mathbb{B}(0;\epsilon)}\inf_{\lambda} \Big(\Psi_1(\lambda+x)-\Psi_1(\lambda_1)+\Psi_2(\lambda+y)-\Psi_2(\lambda_2)\Big)\,,
\end{align*}
the first equality following from shifting the minimization over $(x,y)\in\partial \mathbb{B}(\lambda;\epsilon)\times \partial \mathbb{B}(\lambda;\epsilon)$ to the set of points $(x+\lambda,y+\lambda)$ for $(x,y)\in\partial \mathbb{B}(0;\epsilon)\times \partial \mathbb{B}(0;\varepsilon)$ and the second equality by a simple \emph{min-min} principle. 
The potentials $\Psi_1$ and $\Psi_2$ being uniformly convex, the minimizers of $\lambda\mapsto\Psi_1(\lambda+x)-\Psi_1(\lambda_1)+\Psi_2(\lambda+y)-\Psi_2(\lambda_2)$ are explicitly given by 
\[\lambda(x,y)=\Big(\nabla \Psi_1(\cdot+x)+\nabla \Psi_2(\cdot+y)\Big)^{-1}(0)\,,
\]
independently of $(x,y)$. As such, $\inf_\lambda H_\varepsilon(\lambda)$ rewrites as 
\[
\inf_{(x,y)\in\partial \mathbb{B}(0;\epsilon)\times \partial \mathbb{B}(0;\epsilon)}\Bigg(\Psi_1(\lambda(x,y)+x)-\Psi_1(\lambda_1)+\Psi_2(\lambda(x,y)+y)-\Psi_2(\lambda_2)
\Bigg)\,,
\]
and, subsequently, any $\lambda_\epsilon$ of $\mathcal M_{\epsilon}$ corresponds to a point $\lambda_\varepsilon(x^*_\epsilon,y^*_\epsilon)$ where $(x^*_\epsilon,y^*_\epsilon)$ achieves the above minimum. Compared to the limit collision-location, $\lambda_0=(\nabla \Psi_1+\nabla \Psi_2)^{-1}(0)$, the minimizers $\lambda_\varepsilon$ are so perturbations of $\lambda_0$ in a direction of magnitude $\varepsilon$. As $\varepsilon$ decreases to~$0$, the regularity of $\Psi_1$ and $\Psi_2$ guarantees that $\mathcal M_\varepsilon$  concentrates on the single point~$\{\lambda_0\}$ and $\inf_\lambda h_\varepsilon$ converges naturally to $\inf_\lambda h_0$ respectively as $\epsilon\downarrow 0$. Illustratively, consider the case where $\Psi_1$ and $\Psi_2$ are quadratic potentials of the form $\Psi_i(z)=\gamma_i||z-\lambda_i||^2/2$ for~$\gamma_i>0$. The first collision-location $\lambda_0$ is then explicitly given by $\frac{\gamma_1\lambda_1+\gamma_2\lambda_2}{\gamma_1+\gamma_2}$ and the exit-cost of the first collision by $\underline{\h}_0=\frac{\gamma_1\gamma_2}{2(\gamma_1+\gamma_2)}||\lambda_2-\lambda_1||^2$. Meanwhile, for any $(x,y)$, $\lambda(x,y)$ is given by~$\frac{\gamma_1(\lambda_1-x)+\gamma_2(\lambda_2-y)}{\gamma_1+\gamma_2}$. 
\end{rem}

\section{On the first collision of two self-stabilizing processes}\label{sec:SelfStabilizingCase}

Throughout this section, we consider the pair of self-stabilizing processes \eqref{MV1} and \eqref{MV2}, which under $(\mathbf{A})$ formulates as 
\begin{subequations}
\label{SSD}
\begin{align}
X_t&=x_1+\sigma B_t-\int_0^t\nabla V(X_s)ds-\alpha \int_0^t\left(X_s-\mathbb E[X_s]\right)ds,\,t\ge 0,\label{SSD1}\\
Y_t&=x_2+\sigma \widetilde{B}_t-\int_0^t\nabla V(Y_s)ds-\alpha\int_0^t\left(Y_s-\mathbb E[Y_s]\right)ds,\,t\ge 0,\label{SSD2}
\end{align}
\end{subequations}
and establish Theorem \ref{thm:main1}. A first step towards this aim consists in obtaining a coupling estimate between \eqref{SSD1} and \eqref{SSD2}, and the ``linearized'' flows:
\begin{subequations}
\begin{equation}
\label{x}
x_{T,t}^\sigma=X_{T}+\sigma\big(B_t-B_{T}\big)-\int_{T}^t\nabla V(x_{T,s}^\sigma)ds-\alpha\int_{T}^t(x_{T,s}^\sigma-\lambda_1)ds\,,t\ge T,
\end{equation}
and
\begin{equation}
\label{y}
y_{T,t}^\sigma=Y_{T}+\sigma\big(\widetilde{B}_t-\widetilde{B}_{T}\big)-\int_{T}^t\nabla V(y_{T,s}^\sigma)ds-\alpha\int_{T}^t\Big(y_{T,s}^\sigma-\lambda_2\Big)ds\,,t\ge T.
\end{equation}
\end{subequations}
The starting time $T$ from which the coupling is constructed will be specified in a few lines. 

Following \cite[Theorem 3.4]{Kinetic}, the condition of synchronization $(\mathbf{A})-(iii)$ ensures that the mean vector~$(\mathbb E[X_t],\mathbb E[Y_t])$ and $(\lambda_1,\lambda_2)$ can be found arbitrarily close at large time $t$. Specifically: for any $\kappa>0$, there exist $T_\kappa$, finite and independent of $\sigma$, and $0<\sigma_\kappa<\infty$ such that:
\begin{equation}
\label{carl_bis}
\sup_{\sigma<\sigma_\kappa}\sup_{t\geq T_\kappa}||\mathbb E[X_t]-\lambda_1||+\sup_{\sigma<\sigma_\kappa}\sup_{t\geq T_\kappa}||\mathbb E[Y_t]-\lambda_2||\leq \kappa\,.
\end{equation}
This successively leads to the following coupling lemma:
\begin{lem}
\label{lucille} For any $\xi>0$, there exists $T_\xi\in(0,\infty)$, depending only on $\alpha,\theta$ and $\xi$ such that
\begin{equation*}
\PP\left\{\sup_{t\geq T_\xi}||X_t-x_{T_\xi,t}^\sigma||\geq\xi\right\}=0=
\PP\left\{\sup_{t\geq T_\xi}||Y_t-y_{T_\xi,t}^\sigma||\geq\xi\right\}\,.
\end{equation*}
\end{lem}
\begin{rem} This statement slightly extends the coupling estimate previously obtained in \cite[Lemma 4.6]{JOTP}. There, the coupling between $X$ and $x^\sigma$ was established on an interval of the form $[T_\kappa,T^s_\kappa(\sigma)]$ where $T^s_\kappa(\sigma)$ is the first time such that $\mathbb P\{\tau_D\le T^s_\kappa(\sigma)\}=\kappa$ for~$D$ a stable set by $x\mapsto -\nabla V(x)-\alpha(x-\lambda_1)$ containing $\lambda_1$ in its interior. Here, $\tau_D$ does correspond to the first exit-time of the sole diffusion $X$ from $D$. The absence of a limiting upper-time horizon  in Lemma \ref{lucille} is justified by additionally following the arguments of~\cite[Lemma~4.1]{Kinetic}.  
\end{rem}
  As a preliminary step for the proof of Lemma \ref{lucille}, we should recall the following key result: 
\begin{lem}[\cite{BRTV}, Lemma 3.7]
\label{randal}
Let $f:\mathbb R^+\rightarrow \mathbb R$ be a continuous and differentiable function. Assuming that there exists $l>0$ such that $\{t>0 : f(t)> l\}\subset \{t>0 : f'(t)< 0\}$, then $f(t)\leq \max(f(0),l)$ for all $t$.
\end{lem}
\begin{proof}[Proof of Lemma \ref{lucille}] We only focus the demonstration on the pair $X$ and $x^\sigma$, the coupling between $Y$ and $y^\sigma$ being handled in the exact same way. Fix $\xi>0$ and, given~$\kappa>0$~-~to be chosen later - and $\sigma>0$, let $T_\kappa$ be given by \eqref{carl_bis}. For any $t\ge T_\kappa$,  
\begin{equation*}
d||X_t-x_{T_\kappa,t}^\sigma||^2=-2\left(X_t-x_{T_\kappa,t}^\sigma\right)\bigg(\nabla W_{\mu_t}\left(X_t\right)-\nabla W_{\lambda_1}\left(x_{T_\kappa,t}^\sigma\right)\bigg)dt\,,
\end{equation*}
recalling that $\mu_t=\mathcal{L}\left(X_t\right)$ and setting $W_\mu(x):=V(x)+F\ast\mu(x)$ and $W_{\lambda_1}(x):=W_{\delta_{\lambda_1}}(x)$ where $\delta_{\lambda_1}$ is the Dirac measure in $\lambda_1$. Adding and subtracting $\nabla W_{\lambda_1}(X_t)$, the above yields
\begin{align*}
&d||X_t-x_{T_\kappa,t}^\sigma||^2\\
&=-2\left(X_t-x_{T_\kappa,t}^\sigma\right)\left(\nabla W_{\lambda_1}\left(X_t\right)-\nabla W_{\lambda_1}\left(x_{T_\kappa,t}^\sigma\right)\right)dt\\
&-2\left(X_t-x_{T_\kappa,t}^\sigma\right)\left(\nabla W_{\mu_t}\left(X_t\right)-\nabla W_{\lambda_1}\left(X_t\right)\right)dt\,,\\
&=-2\left(X_t-x_{T_\kappa,t}^\sigma\right)\left(\nabla W_{\lambda_1}\left(X_t\right)-\nabla W_{\lambda_1}\left(x_{T_\kappa,t}^\sigma\right)\right)dt-2\alpha\left(X_t-x_{T_\kappa,t}^\sigma\right)
\left(\lambda_1-\mathbb E[X_t] \right)
dt\,.
\end{align*}
As $(A)-(iii)$ implies $x\mapsto W_{\lambda_1}(x)
=V(x)+ F(x-\lambda)$ is $(\alpha+\theta)$-convex, $\zeta_t:=||X_t-x_{T_\kappa,t}^\sigma||^2$ is differentiable and satisfies:
\begin{equation*}
\frac{d}{dt}\zeta_t\leq-2(\alpha+\theta)\zeta_t+2\alpha\sqrt{\zeta_t}||\lambda_1-\mathbb E[X_t]||\,.
\end{equation*}
As \eqref{carl_bis} implies $||\lambda_1-\mathbb E[X_t]||\leq\sqrt{\EE[||\lambda_1-X_t||^2]}\leq\kappa$, it follows that
\begin{equation*}
\frac{d}{dt}\zeta_t\leq-2(\alpha+\theta)\zeta_t+2\alpha\sqrt{\kappa}\sqrt{\zeta_t}\,.
\end{equation*}
Applying Lemma~\ref{randal} and since $\zeta_{T_\kappa}=0$, it follows that $\zeta_t\le \big(\frac{\alpha}{\alpha+\theta}\kappa\big)^2$. Taking next $\kappa<\frac{\alpha+\theta}{\alpha}\xi$ gives the claim.
\end{proof}
\noindent
Since the potentials 
\begin{equation}
\label{GlobalPot}
\Psi_1(x):=V(x)+\frac{\alpha}{2}||x-\lambda_1||^2,\,\quad\Psi_2(y):=V(y)+\frac{\alpha}{2}||y-\lambda_2||^2\,,
\end{equation}
driving \eqref{x} and \eqref{y} are uniformly convex, applying Lemma \ref{dale-bis} in Section \ref{sec:LinearCase} - up to a time shift, allowed by the Markov property of $(x^\sigma,y^\sigma)$ - yields to the Kramers' type law:

\begin{equation}
\label{merlemerle}
\lim_{\sigma\to0}\PP\left\{\exp\left[\frac{2}{\sigma^2}\left(\widehat{H}_\varepsilon(\lambda)-\delta\right)\right]<\tau_{\lambda,\varepsilon}(\sigma)<\exp\left[\frac{2}{\sigma^2}\left(\widehat{H}_\varepsilon(\lambda)+\delta\right)\right]\right\}=1\,,
\end{equation}
where
$$\tau_{\lambda,\varepsilon}(\sigma):=\inf\left\{t\geq T_\xi\,\,:\,\,(x_{T_\xi,t}^\sigma,y_{T_\xi,t}^\sigma)\in\mathbb{B}\left(\lambda;\varepsilon\right)\times\mathbb{B}\left(\lambda;\varepsilon\right)\right\}\,.$$
Here, $T_\xi$ is given as in Lemma \ref{lucille}, and 
\begin{equation*}
\widehat H_\varepsilon(\lambda)=\left\{
\begin{aligned}
&\inf_{y\in\partial \mathbb B(\lambda;\epsilon)}\big(V(y)-V(\lambda_2)+\frac{\alpha}{2}||y-\lambda_2||^2\big)\,\mbox{if}\,\,||\lambda-\lambda_1||< \epsilon\,,\\
&\inf_{x\in\partial \mathbb B(\lambda;\epsilon)}\big(V(x)-V(\lambda_1)+\frac{\alpha}{2}||x-\lambda_1||^2\big)\,\mbox{if}\,\,||\lambda-\lambda_2||< \epsilon\,,\\
&\inf_{x,y\in \partial \mathbb{B}(\lambda;\varepsilon)}\bigg(V(x)-V(\lambda_1)+V(y)-V(\lambda_2)+\frac{\alpha}{2}||x-\lambda_1||^2+\frac{\alpha}{2}||y-\lambda_2||^2\bigg)\\
&\:\:\:\:\:\:\:\:\:\:\:\text{if}\,\min_i||\lambda-\lambda_i||>\epsilon\,.
\end{aligned}
\right.
\end{equation*}
}

The asymptotic \eqref{merlemerle} and the collision-location property stated in Lemma \ref{dale-bis} - taking also into account the succeeding discussion on the simplification of the minimizer sets - both shift to the McKean-Vlasov system $(X,Y)$ and leads to the following result:

\begin{prop}
\label{torche} The first entering-time of $(X,Y)$ in the domain $\mathbb{B}(\lambda;\varepsilon)\times\mathbb{B}\left(\lambda;\varepsilon\right)$: 
\[
\beta_{\lambda,\varepsilon}(\sigma):=\inf\left\{t\geq0\,\,:\,\,(X_t,Y_t)\in\mathbb{B}\left(\lambda;\varepsilon\right)\times\mathbb{B}\left(\lambda;\varepsilon\right) \right\}\,,
\]
satisfies, for any $\lambda\in\bRb^d$ and any $\delta>0$,

\begin{equation}
\label{merlemerlemerle}
\lim_{\sigma\to0}\PP\left\{\exp\left[\frac{2}{\sigma^2}\left(\widehat{H}_\varepsilon(\lambda)-\delta\right)\right]<\beta_{\lambda,\varepsilon}(\sigma)<\exp\left[\frac{2}{\sigma^2}\left(\widehat{H}_\varepsilon(\lambda)+\delta\right)\right]\right\}=1\,,
\end{equation}
and 
\begin{equation*}
\lim_{\sigma\to0}\PP\left\{{\rm dist}\bigg((X_{\widehat \beta_{\lambda,\varepsilon}(\sigma)},Y_{\widehat \beta_{\lambda,\varepsilon}(\sigma)}),\mathbb B(\lambda;\varepsilon)\times\mathbb B(\lambda;\varepsilon)\bigg)\le \delta\right\}=1\,.
\end{equation*}
\end{prop}

\begin{proof}Fix $\delta>0$ and let $\sigma$ and $\xi$ be small enough so that $T_\xi$ given in \eqref{carl_bis} is smaller than~$\exp\{\frac{2}{\sigma^2}(\widehat{H}_\varepsilon(\lambda)-\delta)\}$. As  $\lim_{\sigma\rightarrow 0}\beta_{\lambda,\varepsilon}(\sigma)=\infty$ a.s., the event $\{T_\xi> \beta_{\lambda,\varepsilon}(\sigma)\}$ becomes negligible at the limit $\sigma\downarrow 0$. On the remaining event $\{T_\xi\le \beta_{\lambda,\varepsilon}(\sigma)\}$, according to Lemma~\ref{lucille}, $(X_t,Y_t)$ and $(x_{T_\xi,t}^\sigma,y_{T_\xi,t}^\sigma)$ are at distance of at most $\xi$ from each others. This way, for $\xi<\varepsilon$, $\widehat{\beta}_{\lambda,\varepsilon}(\sigma)$ necessarily lies in the interval $[\tau_{\lambda,\varepsilon+\xi}(\sigma),\tau_{\lambda,\varepsilon-\xi}(\sigma)]$. As $\eta\mapsto \widehat{H}_\eta(\lambda)$ is continuous, we can further choose $\xi$ again small enough  so that $H_\varepsilon(\lambda)+\delta\ge \widehat{H}_{\varepsilon-\xi}(\lambda)+\delta'$ and $\widehat{H}_\varepsilon(\lambda)-\delta\le \widehat{H}_{\varepsilon+\xi}(\lambda)-\delta''$, for some ${\delta'}
,{\delta''}>0$. The Kramers' type law~\eqref{merlemerle} then ensures
\begin{align*}
&\lim_{\sigma\rightarrow 0}\PP\left\{\exp\left[\frac{2}{\sigma^2}\left(\widehat{H}_{\varepsilon}(\lambda)-\delta\right)\right]<\beta_{\lambda,\varepsilon}(\sigma)\right\}\\
&\ge\lim_{\sigma\rightarrow 0}\PP\left\{\exp\left[\frac{2}{\sigma^2}\left(\widehat{H}_{\varepsilon+\xi}(\lambda)-\delta''\right)\right]<\tau_{\lambda,\varepsilon+\xi}(\sigma)\right\}\ =1\,.
\end{align*}
and
\begin{align*}
&\lim_{\sigma\rightarrow 0}\PP\left\{\beta_{\lambda,\varepsilon}(\sigma)<\exp\left[\frac{2}{\sigma^2}\left(\widehat{H}_{\varepsilon}(\lambda)+\delta\right)\right]\right\}\\
&\ge \lim_{\sigma\rightarrow 0}\PP\left\{\tau_{\lambda,\varepsilon-\xi}(\sigma)<\exp\left[\frac{2}{\sigma^2}\left(\widehat{H}_{\varepsilon-\xi}(\lambda)+\delta'\right)\right]\right\}=1\,.
\end{align*}
The asymptotic of $(X_{\beta_{\lambda,\varepsilon}(\sigma)},Y_{\beta_{\lambda,\varepsilon}(\sigma)})$ is an immediate consequence of the very definition of~$\beta_{\lambda,\varepsilon}(\sigma)$.
\end{proof}

As a consequence of the above, we immediately deduce the asymptotic of the first time that the diffusions $X$ and $Y$ are at a distance $2\varepsilon$.

\begin{prop}
\label{torche2} For $\Psi_1$ and $\Psi_2$ as in~\eqref{GlobalPot}, define 
\begin{equation}\label{epsiloncollisioncost}
H_\varepsilon(\lambda):=\inf_{x\in \partial \mathbb{B}(\lambda;\varepsilon)}\big(\Psi_1(x)-\Psi_1(\lambda_1)\big)+\inf_{y\in \partial \mathbb{B}(\lambda;\varepsilon)}\big(\Psi_2(y)-\Psi_2(\lambda_2)\big)\,,
\end{equation} 
and  the threshold 

\begin{equation}\label{ColliRad2}\varepsilon_c=\inf\left\{\varepsilon\le \varepsilon_0\,:\,\inf_{\lambda\in\mathbb B(\lambda_i;\varepsilon)}\inf_{z\in\partial \mathbb B(\lambda;\varepsilon)}\Big(\Psi_j(z)-\Psi_j(\lambda_j)\Big)=\inf_\lambda H_\varepsilon(\lambda),\,\,i\neq j\in\{1,2\}\right\}\,.
\end{equation}

Let $C_\varepsilon(\sigma)$ be as in \eqref{berlin-bogota}, $\varepsilon\in(0,\varepsilon_c)$, $\underline H_\varepsilon=\min H_\varepsilon$ and let $\mathcal M_\varepsilon$ be the set of minimizer of~$H_\varepsilon$. Then, for any $\delta>0$, 
\begin{equation*}
\lim_{\sigma\to0}\PP\left\{\exp\left[\frac{2}{\sigma^2}\left(\underline H_\varepsilon-\delta\right)\right]<{C}_\varepsilon(\sigma)<\exp\left[\frac{2}{\sigma^2}\left(\underline H_\varepsilon+\delta\right)\right]\right\}=1\,,
\end{equation*}
and
\begin{equation*}
\lim_{\sigma\to0}\PP\left\{\inf_{\lambda_\varepsilon\in\mathcal M_\varepsilon}\max\bigg({\rm dist}\big(X_{C_\varepsilon(\sigma)},\mathbb B(\lambda_\varepsilon;\varepsilon)\big),{\rm dist}\big(Y_{C_\varepsilon(\sigma)},\mathbb B(\lambda_\varepsilon;\varepsilon)\big)\bigg)\ge \delta\right\}=0\,.
\end{equation*}
\end{prop}
\begin{proof}The assumption $\varepsilon<\varepsilon_c$ immediately ensures $\inf_{\lambda}\widehat H_\epsilon(\lambda)=\inf_{\lambda} H_\epsilon(\lambda)$,
and the proof is readily  similar to the one of Proposition \ref{lacollision} replacing $z_t^\sigma$ by the barycenter $Z_t:=2^{-1}(X_t+Y_t)$. 
\end{proof}

In the same way Proposition~\ref{lacollision} yielded Theorem~\ref{lacollisionthm}, Proposition \ref{torche2} yields to Theorem~\ref{thm:main1}.

\section{On the first collision of the particle systems}\label{sec:ParticleCase}

In this section, we establish the Kramers' type law driven the asymptotic of the first collision-time of the mean-field interacting particle systems :
\begin{subequations}
\begin{equation}\label{MFSP1}
X^{i,N}_t=x_1+\sigma B^i_t-\int_0^t \nabla V(X^{i,N}_s)\,ds-\frac{\alpha}{N}\sum_{j=1}^N\int_0^t (X^{i,N}_s-X^{j,N}_s)\,ds\,,\,t\ge 0,
\end{equation}
and
\begin{equation}\label{MFSP2}
Y^{i,N}_t=x_2+\sigma \widetilde B^i_t-\int_0^t \nabla V(Y^{i,N}_s)\,ds-\frac{\alpha}{N}\sum_{j=1}^N\int_0^t (Y^{i,N}_s-Y^{j,N}_s)\,ds\,.\,t\ge 0.
\end{equation}
\end{subequations}
The assumption $(\mathbf{A})-(iv)$ imposes that each pair of particles, $X^{i,N}$ and $Y^{i,N}$, is attracted to the two wells of $V$. As in Section \ref{sec:SelfStabilizingCase}, we also still assume that $x_1$ and $x_2$, as well as $\lambda_1$ and~$\lambda_2$, are at a distance at least $2\varepsilon_0$ from each others.

As a preliminary remark, let us point out that the systems \eqref{MFSP1} and \eqref{MFSP2} can also be equivalently formulated in the form of stochastic gradients flows, by setting $\mathbf X^N_t:=(X^{1,N}_t,\cdots,X^{N,N}_t)$ and $\mathbf Y^N_t:=(Y^{1,N}_t,\cdots,Y^{N,N}_t)$:
\[
\mathbf X^{N}_t=\mathbf{x}_1^N+\sigma \mathbf B_t-\int_0^t \nabla\Upsilon_N(\mathbf X^N_s)\,ds,\,t\ge 0\,,
\]
\[
\mathbf Y^{N}_t=\mathbf{x}_2^N+\sigma \widetilde{\mathbf B}_t-\int_0^t \nabla\Upsilon_N(\mathbf Y^N_s)\,ds,\,t\ge 0\,,
\]
where $\mathbf B:=(B^1,\cdots,B^N)$ and $\widetilde{\mathbf B}:=(\widetilde{B}^1,\cdots,\widetilde{B}^N)$ define independent $\mathbb R^{dN}$-Brownian motions and the driving potential is given by
\[
\Upsilon_N:\mathbf x^N=(x_1,\cdots,x_N)\in\mathbb R^{dN}\mapsto \Upsilon_N(\mathbf x^N)=\sum_{i=1}^NV(x_i)+\frac{\alpha}{4N}\sum_{i,j=1}^N||x_i-x_j||^2.
\]
For $\overline{\mathbf{x}}^N:=\frac{1}{N}\sum_{j=1}^Nx_j$ the empirical mean of $\mathbf x^N$, one can easily check that $\Upsilon^N$ rewrites into
\[
\Upsilon_N(x^N)=\sum_{i=1}^NV(x_i)+\frac{\alpha}{2N}\sum_{i=1}^N ||x_i-\overline{\mathbf{x}}^N||^2\,.
\] 
Roughly, the proof steps to establish Theorem \ref{thm:main2} are similar to the ones leading to Theorem \ref{thm:main1}. Namely, observing that $C^i_{\varepsilon,N}(\sigma)$ can be equivalently reformulated into
\[
C^i_{\varepsilon,N}(\sigma)=\inf_{\lambda\in\mathbb R^d}\beta^i_{\lambda,\varepsilon,N}(\sigma)\,,
\]
\[
\beta^i_{\lambda,\varepsilon,N}(\sigma)=\inf\{t>0\,:\,(X^{i,N}_t,Y^{i,N}_t)\in\mathbb B(\lambda;\epsilon)\times\mathbb B(\lambda;\epsilon)\},
\]
after establishing a coupling estimate between the interacting particles, and their ``linear'' versions (see \eqref{xi} and \eqref{yi} below), we extract a first Kramers' type law for $\beta^i_{\lambda,\varepsilon,N}(\sigma)$. At this step, the related exit-cost is given by
 \begin{align*}
&\inf_{\mathbf{x}^N\in \partial B^{i,N}(\lambda;\varepsilon)}\Upsilon_N(\mathbf x^N)-\Upsilon_N(\lambda_1,\cdots,\lambda_1)+\inf_{\mathbf{y}^N\in \partial B^{i,N}(\lambda;\varepsilon)}\Upsilon_N(\mathbf y^N)-\Upsilon_N(\lambda_2,\cdots,\lambda_2),
\end{align*}
for
\[
B^{i,N}(\lambda;\varepsilon)=\left\{\mathbf x^N=(x_1,...,x_N)\in\mathbb R^{dN}\,;\,x_i\in \mathbb B(\lambda;\varepsilon)\right\}\,.
\]
Following, we will derive the  Kramers' type law of $\inf_{\lambda\in\mathbb R^d}\beta^i_{\lambda,\varepsilon,N}(\sigma)$, replicating the proof arguments of Proposition~\ref{lacollision} and next deducing Theorem~\ref{thm:main2}. The inherent difficulty here relies on establishing coupling estimates which, for the case of self-stabilizing diffusions, was essentially a consequence of establishing \eqref{carl_bis}. For the particle systems \eqref{MFSP1}-\eqref{MFSP2}, we need to establish an analog of~\eqref{carl_bis}, dealing with the empirical means $\overline{X}^N_t=\frac{1}{N}\sum_{j=1}^NX^{j,N}_t$ and~$\overline{Y}^N_t=\frac{1}{N}\sum_{j=1}^NY^{j,N}_t$, in place of the continuous averages. One may argue that the coupling could be achieved by using some propagation of chaos properties of \eqref{MFSP1}-\eqref{MFSP2} towards \eqref{SSD1}-\eqref{SSD2}. This argument could be executed if uniform-in-time propagation of chaos was holding true. However, under our current assumptions, this is not the case. Indeed, the uniform convexity stated in $(\mathbf{A})-(i)$ implies $-\nabla V$ is ``contractively'' one-sided Lipschitz at infinity, namely: for some $R''>0$ large enough
  \[
  -\left( \nabla V(x)-\nabla V(y)\right)(x-y)\leq -\lambda ||x-y||^2, \forall x,y\in\mathbb R^d\,\text{such that}\,||x||,||y||\geq R''\,.
  \]
Combined with its local Lipschitz property, $-\nabla V$ is so globally one-sided Lipschitz continuous, and, by classical coupling arguments (e.g. \cite{BRTV}), we recover the propagation of chaos property: for any finite time-horizon $T$, and, for $(X^i,Y^i)$ the copies of \eqref{SSD1} and \eqref{SSD2} generated by $(B^i,\widetilde{B}^i)$,
 \begin{equation*}
\sup_{0\leq t\leq T}\mathbb{E}\left\{||X_t^i-X_t^{i,N}||^2\right\}+\sup_{0\leq t\leq T}\mathbb{E}\left\{||Y_t^i-Y_t^{i,N}||^2\right\}\le \frac {C(T)}{N}\,.
\end{equation*}
Without further assumption on the convexity of $V$, the constant $C(T)$ depends exponentially on $T$. If $V$ was further assumed to be globally convex, the constant would not depend anymore on $T$ (e.g. \cite{CGM}) and, owing to Lemma~\ref{lucille}, the asymptotics obtained in Proposition~\ref{torche} would apply to $\beta^i_{\lambda,\varepsilon,N}(\sigma)$, for $N$ large enough. The non-(global) convexity of $V$ provides some technical difficulties, which prevent to rely on the asymptotics previously obtained for $(X,Y)$, and rather impose to exhibit analog arguments to the self-stabilizing case. 

Given $0<T<\infty$, define the diffusions $x^{i,\sigma}_T$ and $y^{i,\sigma}_T$ as:
\begin{subequations}
\begin{equation}
\label{xi}
x_{T,t}^{i,\sigma}=X^{i,N}_T+\sigma (B_t^i-B_T^i)-\int_T^t\nabla V(x_{T,s}^{i,\sigma})ds-\alpha\int_T^t(x_{T,s}^{i,\sigma}-\lambda_1)ds\,,\,t\ge T\,,
\end{equation}
and
\begin{equation}
\label{yi}
y_{T,t}^{i,\sigma}=Y^{i,N}_T+\sigma (\widetilde{B}_t^i-\widetilde{B}_T^i)-\int_T^t\nabla V(y_{T,s}^{i,\sigma})ds-\alpha\int_T^t(y_{T,s}^{i,\sigma}-\lambda_2)ds\,,\,t\ge T\,,
\end{equation}
\end{subequations}
and establish a weaker analog to Lemma \ref{lucille}, with 
\begin{prop}
\label{lucille2} For $\varepsilon\in(0,\varepsilon_c)$, let $H_\varepsilon$ be as in \eqref{epsiloncollisioncost} and $\underline H_\varepsilon$ its minima. Then, for any $\lambda\in\mathbb R^d$ and any $\xi>0$, provided $\xi>0$ is small enough, there exists $N_{\varepsilon,\xi}$ and $T_{\varepsilon,\xi}$, both positive, finite and independent of $\sigma$, such that, for any $N\ge N_{\varepsilon,\xi}$,
\begin{align*}
&\lim_{\sigma\rightarrow 0}\PP\left\{\max_{T_{\xi,N}\le t\le \exp[\frac 2{\sigma^2}(\underline H_\varepsilon+2)]}||X_t^{i,N}-x_{T_{\xi,N},t}^{i,\sigma}||\geq\xi\right\}\\
&=\lim_{\sigma\rightarrow 0}\PP\left\{\max_{T_{\xi,N}\le t\le \exp[\frac 2{\sigma^2}(\underline H_\varepsilon+2)]}||Y_t^{i,N}-y_{T_{\xi,N},t}^{i,\sigma}||\geq\xi\right\}=0\,.
\end{align*}
 \end{prop}
Compared to Lemma \ref{lucille}, Proposition \ref{lucille2} only ensures the coupling between $(X^{i,N},Y^{i,N})$ and $(x^{i,\sigma}_T,y^{i,\sigma}_T)$ is effective almost surely at the limit $\sigma\downarrow 0$ and over an interval restrained by a referential right-hand limit $\exp[\frac 2{\sigma^2}(\underline H_\varepsilon+2)]$. While this is enough to carry the same procedure as in Section~\ref{sec:SelfStabilizingCase} and derive the Kramers' type law for $C^i_{\varepsilon,N}(\sigma)$, the reason for these limitations are rather intuitive. As the empirical means in \eqref{MFSP1}-\eqref{MFSP2} are not deterministic,  for $\sigma>0$ arbitrary, $\overline{X}^N$ and $\overline{Y}^N$ wander far away from $\lambda_1$ and $\lambda_2$ at large (but finite) time. As the noises elapse, these events become negligible.  This is why we need to restrict ourselves to a characteristic finite time interval where $\overline{X}^N$ and $\overline{Y}^N$ are arbitrarily close to~$\lambda_1$ and~$\lambda_2$. Since it is sufficient for this time to be strictly larger than $\exp[2\underline H_\varepsilon/{\sigma^2}]$, we choose $\exp[2(\underline H_\varepsilon+2)/{\sigma^2}]$ as a possible upper-bound.

\begin{lem}\label{carl2_reborn} For $\kappa>0$, let $\tau^N_\kappa(\sigma)$ be the first time the diffusion $(\mathbf{X}^{N},\mathbf{Y}^{N})$
exits the domain $\mathcal B^N(\lambda_1^N;\kappa)\times \mathcal B^N(\lambda_2^N;\kappa)$ for  
\[
\mathcal B^N(\lambda_i^N;\kappa):=\{\mathbf x^N\in\mathbb R^{dN}\,:\,||\frac 1{N}\sum_{j=1}^N x_j-\lambda_i||<\kappa\},\:\:\:\:\:\:\mathbf{\lambda^N_i}:=(\lambda_i,\cdots,\lambda_i)\,.
\] 
Therefore, for any $\kappa>0$ and for any $\varepsilon\in(0,\varepsilon_0)$, there exists $N_{\varepsilon,\kappa}$ such that for $N\ge N_{\varepsilon,\kappa}$,
\[
\lim_{\sigma\rightarrow 0}\mathbb P\left\{\tau^N_\kappa(\sigma)\le \exp\big[\frac{2}{\sigma^2}(\underline{H}_\varepsilon+2)\big]\right\}=0\,.
\]
\end{lem}
\begin{proof}  As a preliminary stage, let us establish some properties on the exit-cost related to~$\tau^N_\kappa(\sigma)$ with: in \textbf{Step 1}, we check that $\Upsilon^N$ achieves its minimum either on $(\lambda_1,\cdots,\lambda_1)$ or on $(\lambda_2,\cdots,\lambda_2)$; in \textbf{Step 2}, we show that the exit-cost related to $\mathcal B^N(\lambda_i^N;\kappa)$, explicitly given by
\[\inf_{\mathbf x^N\in\partial \mathcal B^N(\mathbf{\lambda_i^N};\kappa)}\Big(\Upsilon^N(\mathbf x^N)-\Upsilon^N(\mathbf{\lambda_i^N})\Big)\,,
\]
 grows to $\infty$ as $N\uparrow \infty$. The third step finally consists in showing that these properties imply the claim.

\noindent
\textbf{Step 1}: Let $\mathbf{x^{N,\star}}=(x^\star_1,...,x^\star_N)$ be an arbitrary minimizer of $\Upsilon^N$. Then, for any $1\le i\le N$,
\[
\nabla_{x_i}\Upsilon^N(\mathbf{x^{N,\star}})=\nabla V(x^\star_i)+\alpha\Bigg(x^\star_i-\frac{1}{N}\sum_{j=1}^Nx^\star_j\Bigg)=0.
\]
The synchronization condition $(\mathbf{A})-(iii)$ ensuring that, for any $m\in\mathbb R^d$, $x\in\mathbb R^d\mapsto V(x)+F(x-m)$ is uniformly convex, the components $x^\star_i$ are identically given by
\[
x_i^\star=\Big(\nabla V+\alpha{\rm Id}\Big)^{-1}\Bigg(\frac{\alpha}{N}\sum_{j=1}^Nx^\star_j\Bigg).
\]
Necessarily $\frac 1{N}\sum_{j=1}^Nx^\star_j=x_1^\star$ and
 $\nabla \Upsilon^N(\mathbf{x^{N,\star}})$ reduces to $\big(\nabla V(x^\star_1),...,\nabla V(x^\star_1)\big)$. Consequently, the {\bf minimizers} of $\Upsilon^N$ correspond either to $(\lambda_1,\cdots,\lambda_1)$ or $(\lambda_2,\cdots,\lambda_2)$.

\noindent
\textbf{Step 2}: Let us establish that, for any $\kappa$ such that $\kappa<||\lambda_1-\lambda_2||$,
\begin{equation}
\label{michonne}
\lim_{N\to\infty}\inf_{\mathbf{z}^N\in\partial\mathcal B^N(\mathbf{\lambda_i^N};\kappa)}\left(\Upsilon^N(\mathbf{z}^N)-\Upsilon^N(\mathbf{\lambda_i}^N)\right)=+\infty\,.
\end{equation}
 Take $\underline{\lambda}$ as either $\lambda_1$ or $\lambda_2$,  define $\mathbf{\underline{\lambda}^N}=(\underline{\lambda},\cdots,\underline{\lambda})$, and let $\mathbf{\lambda^{N,*}}$ be a minimizer of $\mathbf{z}^N\mapsto\Upsilon^N(\mathbf z^N)-\Upsilon^N(\mathbf{\underline{\lambda}^N})$ under the constraint $\mathbf z^N\in\partial\mathcal B^N(\mathbf{\underline{\lambda}^N};\kappa)$. Then, $\mathbf{\lambda^{N,*}}$ is characterized by the $\mathbb R^{dN}$ differential equation:
\[
\nabla_{\mathbf \lambda_N}\Upsilon^N(\mathbf{\lambda^{N,*}})+L_{N,\kappa}\nabla_{\mathbf \lambda_N}\Bigg(\Big{|}\Big{|}\frac{1}{N}\sum_{j=1}^N\mathbf \lambda^{N,*}_j-\underline{\lambda}\Big{|}\Big{|}^2-\kappa^2\Bigg)=0,
\]
where $L_{N,\kappa}\in\mathbb R$ is the Lagrangian multiplier related to the constraint. Component by component, the equation yields: for all $1\le i\le N$,
\[
\nabla V(\mathbf \lambda^{N,*}_i)+\alpha\Bigg(\mathbf \lambda^{N,*}_i-\frac{1}{N}\sum_{j=1}^N\mathbf \lambda^{N,*}_j\Bigg)+\frac{2L_{N,\kappa}}{N}\Bigg(\frac{1}{N}\sum_{j=1}^N\mathbf \lambda^{N,*}_j-\underline{\lambda}\Bigg)=0.
\]
As $z\mapsto \nabla^2 V(z)+\alpha{\rm Id}$ is positive definite, $z\mapsto \Big(\nabla V(z)+\alpha z\Big)$ is invertible.
Therefore the~$\mathbf \lambda^{N,*}_i$ are all equals to a common value $\lambda^*$ satisfying $\nabla V(\lambda^*)=-\frac{2 L_{N,\kappa}}{N}(\lambda^*-\underline{\lambda})$. Consequently, the barycenter $\frac{1}{N}\sum_{j=1}^N\lambda^{*,N}_j$ reduces to $\lambda^*$ and so
 \[
 \Upsilon^N(\mathbf{\lambda^{*,N}})-\Upsilon^N(\underline{\lambda}^N)= \sum_{i=1}^N \Big(V(\lambda^{*,N}_j)-V(\underline{\lambda})\Big)=N\Big(V(\lambda^*)-V(\underline{\lambda})\Big)\,.
 \]
 Since the constraint $||\frac{1}{N}\sum_{j=1}^N\lambda^{N,*}_j-\underline{\lambda}||^2=\kappa^2$ also reduces to $||\lambda^*-\underline{\lambda}||^2=\kappa^2$, this leads to the lower-bound
 \begin{align*}
 \inf_{\mathbf z^N\in\partial\mathbb{B}^N(\mathbf{\underline{\lambda}^N};\kappa)}\left(\Upsilon^N(\mathbf{z}^N)-\Upsilon^N(\underline{\lambda}^N)\right)
 \ge N \inf_{z\in\mathbb R^d\,:\,||z-\underline{\lambda} ||=\kappa}\Big(V(z)-V(\underline{\lambda})\Big)\,.
 \end{align*}
As long as $\kappa> 0$, \eqref{michonne} follows immediately.

\noindent
\textbf{Final step}:  We remark that $\tau^{N}_{\kappa}(\sigma)$ is a.s. smaller than the first time that one particle, say $X^{1,N}$, exits from $\mathbb{B}(\lambda_1;N\times\kappa)$. The latter being finite a.s., the same holds for $\tau^{N}_{\kappa}(\sigma)$. Define next the (descending) flows:
\[
\mathbf \Phi^{N,-}_t(\mathbf x^N)=\mathbf x^N-\int_0^t \nabla \Upsilon^N(\mathbf \Phi^{N,-}_s(\mathbf x^N))\,ds,\,t\ge 0,\,\mathbf x^N\in\mathbb R^{dN}\,,
\]
and the (truncated) basin of attraction: 
\[
\mathcal D^N=\{\mathbf x^N\in\mathbb R^{dN}\,:\,\Upsilon^N(\mathbf x^N)-\Upsilon^N(\lambda_1^N)\le \underline{H}_\varepsilon+2\,\text{and}\, \lim_{t\rightarrow+\infty}\mathbf \Phi^{N,-}_t(\mathbf x^N)=\mathbf{\lambda_1^N}\}\,.
\] 
By construction, the basin $\mathcal D^N$ is stable by $-\nabla\Upsilon^N$, contains $\mathbf{\lambda_1^N}$, and 
has a related exit-cost given by $\inf_{\mathbf z^N\in\partial \mathcal D^N}\Upsilon^N(\mathbf{z}^N)-\Upsilon^N(\mathbf{\lambda_1^N})$ which is simply $\underline{H}_\varepsilon+2$.
Choosing $N$ large enough,~$\mathcal D^N$ can be forced to lie within $\mathcal B^N(\mathbf{\lambda_1^N};\kappa)$. Indeed, assume that $\mathbf x^N$ is in $\mathcal D^N\setminus \mathcal B^N(\mathbf{\lambda_1^N};\kappa)$. Since $\mathbf{\lambda_1^N}$ is in $\mathcal B^N(\mathbf{\lambda_1^N};\kappa)$, there exists $t>0$ such that 
${\mathbf {y}^N}:=\mathbf \Phi^{N,-}_t(\mathbf x^N)$ lies in the boundary~$\partial\mathcal B^N(\mathbf{\lambda_1^N};\kappa)$. According to \textbf{Step 2}, there exists $N_{\epsilon,\kappa}$ such that for any $N\ge N_{\epsilon,\kappa}$, 
\[
\inf_{{\mathbf z^N}\in \partial\mathbb{B}^N(\mathbf{\underline{\lambda}^N};\kappa)}\Upsilon^N(\mathbf z^N)-\Upsilon^N(\lambda_1^N)>\underline{H}_\varepsilon+2\,.
\]
Subsequently, $\Upsilon^N(\mathbf y^N)-\Upsilon^N(\lambda_1^N)> \underline{H}_\varepsilon+2$ implying that $\mathbf y^N$ can not lie in  $\partial \mathcal B^N(\mathbf{\lambda_1^N};\kappa)$. By extension, $\mathbf x^N$ can not lie in $\mathcal D^N\setminus \mathcal B^N(\mathbf{\lambda_1^N};\kappa)$. 

Define now $\tilde \tau^N_{\kappa}=\inf\{t\ge 0\,:\,\mathbf X^N_t\in \mathcal D^N\}$. As $ \mathcal D^N\subset \mathcal B^N(\mathbf{\lambda_1^N};\kappa)$, $\tilde \tau^N_{\kappa}\le \tau_{\kappa}^{N}$. By~Theorem~\ref{thm:KramersDZ}, 
\[
\lim_{\sigma\rightarrow 0}\mathbb P\left\{\tau^{N}_\kappa(\sigma)\le \exp\left[\frac 2{\sigma^2}(\underline{H}_\varepsilon+2)\right]\right\}\le \lim_{\sigma\rightarrow 0}\mathbb P\left\{\tilde \tau^{N}_\kappa(\sigma)\le \exp\left[\frac 2{\sigma^2}(\underline{H}_\varepsilon+2)\right]\right\}=0\,.
\]
This ends the proof.
\end{proof}

Owing to  Lemma \ref{carl2_reborn} and adapting some proof arguments from the proof of Lemma \ref{lucille}, the proof of Proposition~\ref{lucille2} can be carried on.  
\begin{proof}[Proof of Proposition \ref{lucille2}] 
For $N$ fixed, and given $\kappa>0$, let $T_{\kappa,N}$ be the first time the gradient flows $\Psi_t(x)=x-\int_0^t\nabla V(\Psi_s(x))\,ds$ and $\Psi_t(y)=y-\int_0^t\nabla V(\Psi_s(y))\,ds$ are simultaneously at a $\kappa$-neighborhood of their respective attractors, $\lambda_1$ and $\lambda_2$. Owing to the large deviations principle \eqref{LDP} applied to $(\mathbf{X^N},\mathbf{Y^N})$, for all $1\le i\le N$,
\begin{align*}
\lim_{\sigma\rightarrow 0}\mathbb P\left\{\min\left(|| X^{i,N}_{T_{\kappa,N}}-\Psi_{T_{\kappa,N}}(x_1)||,|| Y^{i,N}_{T_{\kappa,N}}-\Psi_{T_{\kappa,N}}(x_2)||\right)\ge \kappa\right\}=0\,,
\end{align*}
which, owing to the exchangeability of the particle systems and applying Jensen's inequality, yields
\begin{align*}
 \lim_{\sigma\rightarrow 0}\mathbb P\left\{\min\left(|| \overline{X}^{i,N}_{T_{\kappa,N}}-\lambda_1||,|| \overline{Y}^{i,N}_{T_{\kappa,N}}-\lambda_2||\right)\ge \kappa\right\}=0\,.
 \end{align*}

With $\kappa$ still arbitrarily positive, choose next $N_{\varepsilon,\kappa}$ as in  Lemma \ref{carl2_reborn} and $\sigma$ small enough so that $T_{\varepsilon,\kappa}:=T_{\kappa,N_{\varepsilon,\kappa}}$ given as above satisfies $T_{\varepsilon,\kappa}<2^{-1}\exp\big[\frac 2{\sigma^2}(\underline{H}_\varepsilon+2)]$. Up to a time shift, Lemma \ref{carl2_reborn} yields, for any $N\ge N_{\varepsilon,\kappa}$,
\[
\lim_{\sigma\rightarrow 0}\mathbb P\left\{\max_{t\in \big[T_{\varepsilon,\kappa},\exp[\frac{2}{\sigma^2}(\underline{H}_\varepsilon+2)]\big]}\Big(||\overline{X}^{N}_t-\lambda_1||+|| \overline{Y}^{N}_t-\lambda_2||\Big)\ge 2\kappa\right\}=0\,.
\]
As such, we can consider the comparison between $X^{i,N}$ and $\left(x^{i,\sigma}_{T_{\varepsilon,\kappa},t}\right)_{t\geq T_{\varepsilon,\kappa}}$ under the event\\
$\left\{\max_{t\in[T_{\varepsilon,\kappa},\exp[\frac 2{\sigma^2}(\underline{H}_\varepsilon+2)]]}||\overline X^N_t-\lambda_1||<\kappa\right\}$. Observe next that, for all $t\ge T_{\varepsilon,\kappa}$, the path difference between $X^{i,N}$ and $\left(x^{i,\sigma}_{T_{\varepsilon,\kappa},t}\right)_{t\geq T_{\varepsilon,\kappa}}$is given by
\begin{align*}
X_t^{i,N}-x_{T_{\varepsilon,\kappa},t}^{\sigma,i}	&=-\int_{T_{\varepsilon,\kappa}}^t\left(\nabla V(X_s^{i,N})+\alpha X_s^{i,N}-\nabla V(x_{T_{\varepsilon,\kappa},s}^{\sigma,i})+\alpha x_{T_{\varepsilon,\kappa},s}^{\sigma,i}\right)\,ds\\
&+\alpha\int_{T_{\varepsilon,\kappa}}^t\left(\overline X^{N}_s-\lambda_1\right)\,ds\,.
\end{align*}
Consequently, by $(\mathbf{A})-(iii)$,
\begin{align*}
&\frac{d}{dt}||X_t^{i,N}-x_{T_{\varepsilon,\kappa},t}^{i,\sigma}||^2\\
&=-2\big(X_t^{i,N}- x_{T_{\varepsilon,\kappa},t}^{i,\sigma}\big)\left(\nabla V(X_t^{i,N})+\alpha X_t^{i,N}-\nabla V(x_{T_{\varepsilon,\kappa},t}^{i,\sigma})+\alpha x_{T_{\varepsilon,\kappa},t}^{i,\sigma}\right)\\
&+2\alpha\big(X_t^{i,N}- x_{T_{\varepsilon,\kappa},t}^{i,\sigma}\big)(\overline X^{N}_t-\lambda_1)\\
& \le -2(\alpha+\theta)||X_t^{i,N}- x_{T_{\varepsilon,\kappa},t}^{i,\sigma}||^2
+2\alpha\big(X_t^{i,N}- x_{T_{\varepsilon,\kappa},t}^{i,\sigma}\big)(\overline X^{N}_t-\lambda_1)\,.
\end{align*}
On the event $\left\{\max_{t\in[T_{\varepsilon,\kappa},\exp[\frac 2{\sigma^2}(\underline{H}_\varepsilon+2)]]}||\overline X^N_t-\lambda_1||<\kappa\right\}$, the above yields
\[
\frac{d}{dt}||X_t^{i,N}-x_{T_{\varepsilon,\kappa},t}^{i,\sigma}||^2 \le 2||X_t^{i,N}- x_{T_{\varepsilon,\kappa},t}^{i,\sigma}||\left(\alpha\kappa-(\alpha+\theta)||X_t^{i,N}- x_{T_{\varepsilon,\kappa},t}^{i,\sigma}||\right)\,.
\]
Applying Lemma \ref{randal}, it follows that $||X_t^{i,N}-x_{T_{\varepsilon,\kappa},t}^{i,\sigma}||\leq\frac{\alpha\kappa}{\alpha+\theta}$ and, taking $\kappa<\frac{\alpha+\theta}{\alpha}\xi$ yields that $||X_t^{i,N}-x_{T_{\varepsilon,\kappa},t}^{i,\sigma}||\leq\xi$. Indeed, at time $T_{\varepsilon,\kappa}$, the two processes are equal.

\noindent 
Applying the same reasoning to $||Y_t^{i,N}-y_{T_{\varepsilon,\kappa},t}^{i,\sigma} ||$, the claim follows.
\end{proof}

\noindent
According to Lemma \ref{dale-bis}, and choosing $\varepsilon<\varepsilon_c$ with 
$\varepsilon_c$ as in \eqref{ColliRad2}, and given $(X_{T_{\varepsilon,\kappa}}^{i,N},Y_{T_{\varepsilon,\kappa}}^{i,N})$, the hitting-times
\[
\widehat{\beta}^i_{\lambda,\varepsilon,N}(\sigma)=\inf\left\{t\ge T_{\varepsilon,\kappa}\,:\,(x^{i,\sigma}_{T_{\varepsilon,\kappa},t},y^{i,\sigma}_{T_{\varepsilon,\kappa},t})\in\mathbb{B}(\lambda;\varepsilon)\times \mathbb{B}(\lambda;\varepsilon)\right\},\,1\le i\le N,
\]
all satisfy the Kramers' type law with the exit-cost $H_\varepsilon(\lambda)$ as in \eqref{merlemerlemerle} and the exit-property:
\[
\lim_{\sigma\rightarrow 0}\mathbb P^{i,N}_{T_{\varepsilon,\kappa};(x,y)}\left\{{\rm dist}\left(\Big(x^{i,\sigma}_{T_{\varepsilon,\kappa},\widehat{\beta}^i_{\lambda,\varepsilon,N}(\sigma)},y^{i,\sigma}_{T_{\varepsilon,\kappa},\widehat{\beta}^i_{\lambda,\varepsilon,N}(\sigma)}\Big),\mathbb B(\lambda;\varepsilon)\times\mathbb B(\lambda;\varepsilon) \right)\leq\delta\right\}=1\,.
\]
for $\mathbb P^{i,N}_{T_{\varepsilon,\kappa};(x,y)}$ the conditional probability given $\{(X_{T_{\varepsilon,\kappa}}^{i,N},Y_{T_{\varepsilon,\kappa}}^{i,N})=(x,y)\}$. 

Following the same proof arguments as for Proposition \ref{torche}, we obtain

\begin{prop}
\label{tyrese}
For any $\varepsilon\in(0,\varepsilon_c)$ and $\lambda\in\mathbb R^d$, provided $N$ is large enough, we have: for any $\delta>0$ and any $1\le i\le N$,

\begin{equation*}
\lim_{\sigma\to0}\PP\left\{\exp\left[\frac{2}{\sigma^2}\left(H_{\varepsilon}(\lambda)-\delta\right)\right]<\widehat{\beta}^{i}_{\lambda,\varepsilon,N}(\sigma)<\exp\left[\frac{2}{\sigma^2}\left(H_{\varepsilon}(\lambda)+\delta\right)\right]\right\}=1\,,
\end{equation*}
and
\[
\lim_{\sigma\rightarrow 0}\mathbb P\left\{{\rm dist}\left(\Big(X^{i,N}_{\widehat{\beta}^{i}_{\lambda,\varepsilon,N}(\sigma)},Y^{i,N}_{\widehat{\beta}^{i}_{\lambda,\varepsilon,N}(\sigma)}\Big),\mathbb B(\lambda;\varepsilon)\times\mathbb B(\lambda;\varepsilon)\right)<\delta \right\}=1\,.
\]\end{prop}

From this, we can next derive the analog of Proposition \ref{torche2}  which enables to conclude Theorem \ref{thm:main2}.

\begin{prop}
\label{ned} Let $H_\varepsilon$, $\underline H_\varepsilon$, $\mathcal M_\varepsilon$ and $\varepsilon_c$ be as in Proposition \ref{torche2}. For any $\varepsilon\in(0,\varepsilon_c)$ and assuming that $N$ is large enough, it holds: for any $\delta>0$,
\begin{equation*}
\lim_{\sigma\to0}\PP\left\{\exp\left[\frac{2}{\sigma^2}\left(\underline{H}_{\varepsilon}-\delta\right)\right]<\mathcal{C}^{i}_{\varepsilon,N}(\sigma)<
\exp\left[\frac{2}{\sigma^2}\left(\underline{H}_{\varepsilon}+\delta\right)\right]\right\}=1\,.
\end{equation*}

Moreover, the collision-location persists near $\mathcal M_\varepsilon$ with: for any $\delta >0$, $1\le i\le N$,
\begin{equation*}
\lim_{\sigma\to0}\PP\left\{\inf_{\lambda_\varepsilon\in\mathcal M_\varepsilon}\max\bigg({\rm dist}\big(X^{i,N}_{C^{i}_{\varepsilon,N}(\sigma)},\mathbb B(\lambda_\varepsilon;\varepsilon)\big),{\rm dist}\big(Y^{i,N}_{C^{i}_{\varepsilon,N}(\sigma)},\mathbb B(\lambda_\varepsilon;\varepsilon)\big)\bigg)\ge \delta\right\}=0\,.
\end{equation*}
\end{prop}
 \section{One-dimensional case}\label{sec:1DCase}

As we briefly mentioned in the introduction of the paper, the one-dimensional case provides a framework which allows to consider the {\it exact} first collision-time between the self-stabilizing systems \eqref{MV1}-\eqref{MV2}:
\[
C(\sigma):=\inf\left\{t\ge 0\,:\,X_t=Y_t\right\}\,,
\] 
and the  exact first collision-time between the particle systems \eqref{Ibis} and \eqref{IIbis}:
\[
C^i_N(\sigma)=\inf\left\{t\ge 0\,:\,X^{i,N}_t=Y^{i,N}_t\right\}\,.
\]
The reduction to $d=1$ first ensures that these collision-times are finite, almost surely. Additionally, while the methodology to derive Kramers' type laws still require  a re-interpretation of the collision-times and a coupling argument relating $(X,Y)$ and $(X^{i,N},Y^{i,N})$~-~with the diffusion $(x^\sigma,y^\sigma)$ defined in \eqref{x}~-~\eqref{y}, the one-dimensional setting enables to significantly simplify the proof arguments exhibited in Sections~\ref{sec:LinearCase}, \ref{sec:SelfStabilizingCase} and \ref{sec:ParticleCase}. Coupling lemmas are notably less significant and the design of suitable enlargements of the collision set, as in Section~\ref{subsec:Linear-collisionA} are not necessary. Such simplifications are allowed by the \emph{a priori} location of the first collision-location. Indeed, assuming the wells are ordered such that $\lambda_1<\lambda_2$, necessarily, by~$(\mathbf A)-(iv)$, $x_1<x_2$, and $C(\sigma)$ simply corresponds to the first time~$t$ where~$X_t\ge Y_t$. As the reciprocal case, $\lambda_1>\lambda_2$, a similar  interpretation holds (the ordering between $X$ and~$Y$ being simply reversed), in addition to the assumptions $(\mathbf A)$ and without loss of generality,~$\lambda_1<\lambda_2$ is set in force from now on. Following this observation, we can formulate the interpretation~$C(\sigma)=\inf_{z\in\mathbb R}C_{z,z}(\sigma)$ for  
\begin{equation}\label{buffer_collision1d}
C_{z_1,z_2}(\sigma):=\inf\left\{t\ge 0\,\,:\,\,X_t\ge z_1,\,Y_t\le z_2\right\}\,,z_1,z_2\in\mathbb R.
\end{equation}
By analogy with Section \ref{sec:LinearCase}, we define the {\color{black}domains} $\mathcal{D}_{z_1}^1:=[z_1;+\infty)$ and $\mathcal{D}_{z_2}^2:=(-\infty;z_2]$, so that 
\[
C_{z_1,z_2}(\sigma)=\inf\left\{t\geq0\,\,:\,\,(X_t,Y_t)\notin\big(\mathbb R\times\mathbb R\big)\setminus\big(\mathcal{D}_{z_1}^1\times\mathcal{D}_{z_2}^2\big)\right\}\,.
\]
As it will become obvious in our proof arguments, only the situation  $z_1>\lambda_1$ and $z_2<\lambda_2$ is relevant, the situations where either $z_1\leq\lambda_1$ or $z_2\geq\lambda_2$, having no any particular interest.  

\noindent
Under this simplification, the domain $\mathbb R\setminus\mathcal{D}_{z_i}^i$ is stable by $x\mapsto -V'(x)-F'(x-\lambda_i)$ and Theorem~\ref{thm:KramersDZ} applies to 
\[
c_{z_1,z_2}(\sigma):=\inf\left\{t\geq0\,\,:\,\,(x_t^\sigma,y_t^\sigma)\notin\big(\mathbb R\times\mathbb R\big)\setminus\big(\mathcal{D}_{z_1}^1\times\mathcal{D}_{z_2}^2\big)\right\}\,,
\]
for
\begin{equation*}
x_t^\sigma=x_1+\sigma B_t-\int_0^t\nabla \Psi_1(x_s^\sigma)\,ds,\,t\geq0\,,
\end{equation*}
and
\begin{equation*}
y_t^\sigma=x_2+\sigma \tilde B_t-\int_0^t\nabla \Psi_2(y_s^\sigma)\,ds,\,t\geq0\,,
\end{equation*}
where $\Psi_1(x):=V(x)+F(x-\lambda_1)$ and $\Psi_2(y):=V(y)+F(y-\lambda_2)$. 
This yields to a Kramers' type law for $c_{z_1,z_2}(\sigma)$ with the exit-cost:
\[
\tilde H_0(z_1,z_2):=\left(\Psi_1(z_1)-\Psi_1(\lambda_1)\right)+\left(\Psi_2(z_2)-\Psi_2(\lambda_2)\right)\,.
\]
Adapting the proof arguments of Proposition \ref{torche}, the asymptotic of $c_{z_1,z_2}(\sigma)$ transfers to~$C_{z_1,z_2}(\sigma)$, yielding to:
\begin{lem}
\label{prison} For any $z_1\geq\lambda_1$ and $z_2\leq\lambda_2$ and for any $\delta>0$, it holds: 
\begin{equation}\label{alberto2}
\lim_{\sigma\rightarrow 0}
\mathbb P\left\{
\exp\Big[
\frac{2}{\sigma^2}(\tilde H_0(z_1,z_2)-\delta)\Big]
<C_{z_1,z_2}(\sigma)<
\exp\Big[ 
\frac{2}{\sigma^2}(\tilde H_0(z_1,z_2)+\delta) \Big]
\right\}=1\,,
\end{equation}
and
\[
\lim_{\sigma\rightarrow 0}\mathbb P\left\{\max\bigg(|X_{c_{z_1,z_2}(\sigma)}-z_1|,|Y_{c_{z_1,z_2}(\sigma)}-z_2|\bigg)\le \delta\right\}=1\,.
\]
\end{lem}
From this preparatory lemma, we derive the following Kramers' type law for $C(\sigma)$.

\begin{thm}
\label{victoire2} Let $\lambda_0$ be the unique minimizer of $H_0$ given by \eqref{collisioncost}. Then,
for any~$\delta>0$,
\begin{equation}
\label{eq:victoire2}
\begin{aligned}
&\lim_{\sigma\to0}\PP\left\{\exp\left[\frac{2}{\sigma^2}\left(H_0(\lambda_0)-\delta\right)\right]<C(\sigma)<\exp\left[\frac{2}{\sigma^2}\left(
H_0(\lambda_0)+\delta\right)\right]\right\}=1\,,
\end{aligned}
\end{equation}
and 
\begin{equation}
\label{eq:victoirebis2}
\lim_{\sigma\to0}\PP\left\{\left|X_{C(\sigma)}-\lambda_0\right|\leq\delta\right\}=1\,.
\end{equation}
\end{thm}
\begin{proof}

 \noindent{}{\bf Step 1.} Let us first prove the upper-tail estimate in \eqref{eq:victoire3}, starting with the inequality
\begin{equation*}
\PP\left\{C(\sigma)\geq\exp\Big[\frac{2}{\sigma^2}\left(H_0(\lambda_0)+\delta\right)\Big]\right\}\le
\PP\left\{C_{\lambda_0,\lambda_0}(\sigma)\geq\exp\Big[\frac{2}{\sigma^2}\left(H_0(\lambda_0)+\delta\right)\Big]\right\}\,,
\end{equation*}
with $C_{\lambda_0,\lambda_0}(\sigma)$ defined as in \eqref{buffer_collision1d}. This inequality simply results from Lemma~\ref{prison}; due to the convexity of $\Psi_1$ and~$\Psi_2$, the function $z\mapsto\Psi_1(z)+\Psi_2(z)-\Psi_1(\lambda_1)-\Psi_2(\lambda_2)$ is decreasing on $(-\infty;\lambda_1]$ and increasing on $[\lambda_2;+\infty)$, the unique minimizer of $\tilde H_0$ is necessarily achieved in $(\lambda_1;\lambda_2)$. As $\tilde H_0(z,z)=H_0(z)$, $\lambda_0$ lies in $(\lambda_1;\lambda_2)$. For any $\rho>0$ sufficiently small, as~$x_1<x_2$ we know that the first time the diffusion~$(X,Y)$ reaches the point $(\lambda_0,\lambda_0)$ necessarily occurs before $(X,Y)$ enters the region~$[\lambda_0+\rho;\infty)\times(-\infty;\lambda_0-\rho]$. Therefore,
\begin{align*}
&\PP\left\{C_{\lambda_0,\lambda_0}(\sigma)\geq\exp\Big[\frac{2}{\sigma^2}\left(H_0(\lambda_0)+\delta\right)\Big]\right\}\\
&\le \PP\left\{C_{\lambda_0+\rho,\lambda_0-\rho}(\sigma)\ge\exp\Big[\frac{2}{\sigma^2}\left(H_0(\lambda_0)+\delta\right)\Big]\right\}\,.
\end{align*}
As $\Psi_1(\lambda_0+\rho)+\Psi_2(\lambda_0-\rho)$ converges to $H_0(\lambda_0)$ as $\rho\downarrow 0$, taking $\rho$ sufficiently small so that~$\Psi_1(\lambda_0+\rho)+\Psi_2(\lambda_0-\rho)\le H_0(\lambda_0)+\frac{\delta}{2}$ implies, by \eqref{alberto2},
\begin{equation*}
\lim_{\sigma\to0}\PP\left\{C_{\lambda_0+\rho,\lambda_0-\rho}(\sigma)\ge\exp\Big[\frac{2}{\sigma^2}\left(H_0(\lambda_0)+\delta\right)\Big]\right\}=0\,,
\end{equation*}
and so
\begin{equation}\label{proof_victoire2_1}
\lim_{\sigma\to0}\PP\left\{C(\sigma)\geq\exp\Big[\frac{2}{\sigma^2}\left(H_0(\lambda_0)+\delta\right)\Big]\right\}=0\,.
\end{equation}
\noindent{}{\bf Step 2. } For the collision-location estimate \eqref{eq:victoirebis2}, let us check that $\lim_{\sigma\rightarrow 0}\mathbb{P}\{|X_{C(\sigma)}-\lambda_0|>\rho\}=0$ for any $\rho>0$. To this aim, we will show that the collision does not persist outside~$[\lambda_1;\lambda_2]$ then, by a compactness argument, we will show that it necessarily occurs within~$(\lambda_0-\rho;\lambda_0+\rho)$.\\
\noindent{}{\bf Step 2.1. } Start with the following inequality:
\begin{align*}
\PP\left\{X_{C(\sigma)}\notin[\lambda_1;\lambda_2]\right\}&\leq\PP\left\{X_{C(\sigma)}\notin[\lambda_1;\lambda_2],C(\sigma)\leq\exp\Big[\frac{2}{\sigma^2}\left(H_0(\lambda_0)+\delta\right)\Big]\right\}\\
&\quad+\PP\left\{C(\sigma)\geq\exp\Big[\frac{2}{\sigma^2}\left(H_0(\lambda_0)+\delta\right)\Big]\right\}\,,
\end{align*}
for some $\delta>0$. Due to \eqref{proof_victoire2_1}, the second term in the r.h.s. tends to $0$ as $\sigma\downarrow 0$. Defining~$\tilde\beta_1(\sigma):=\inf\left\{t\geq0\,\,:\,\,Y_t\leq\lambda_1\right\}$ and~$\tilde\beta_2(\sigma):=\inf\left\{t\geq0\,\,:\,\,X_t\geq\lambda_2\right\}$, we have 

\begin{align*}
&\PP\left\{X_{C(\sigma)}\notin[\lambda_1;\lambda_2],C(\sigma)\leq\exp\Big[\frac{2}{\sigma^2}\left(H_0(\lambda_0)+\delta\right)\Big]\right\}\\
&\leq\PP\left\{\tilde\beta_1(\sigma)\leq\exp\Big[\frac{2}{\sigma^2}\left(H_0(\lambda_0)+\delta\right)\Big]\right\}+\PP\left\{\tilde\beta_2(\sigma)\leq\exp\Big[\frac{2}{\sigma^2}\left(H_0(\lambda_0)+\delta\right)\Big]\right\}\,.
\end{align*}
Choosing $\delta$ small enough so that $\min\left\{\Psi_1(\lambda_2)-\Psi_1(\lambda_1);\Psi_2(\lambda_1)-\Psi_2(\lambda_2)\right\}>H_0(\lambda_0)+\delta$, Theorem~\ref{thm:KramersDZ} yields, for $k=1,2$,

\[
\lim_{\sigma\rightarrow 0}\PP\left\{\tilde\beta_k(\sigma)\leq\exp\Big[\frac{2}{\sigma^2}\left(H_0(\lambda_0)+\delta\right)\Big]\right\}=0\,.
\]
We immediately deduce
\begin{equation*}
\lim_{\sigma\to0}\PP\left\{X_{C(\sigma)}\notin[\lambda_1;\lambda_2]\right\}=0\,.
\end{equation*}

\noindent{}{\bf Step 2.2. } Given {\bf Step 2.1}, we can focus our claim on showing $X_{C(\sigma)}$ does not persist on~$]\lambda_0-\rho;\lambda_0+\rho[^c\cap[\lambda_1;\lambda_2]=[\lambda_1;\lambda_0-\rho]\cup[\lambda_0+\rho;\lambda_2]$. It is further sufficient to check this assertion for the interval $[\lambda_0+\rho;\lambda_2]$, 
the reasoning for the case $[\lambda_1;\lambda_0-\rho]$ being similar. As $[\lambda_0+\rho;\lambda_2]$ is a compact interval, we can write $[\lambda_0+\rho;\lambda_2]\subset\cup_{k=1}^L]\eta_k-r;\eta_k+r[$ - where $r>0$ will be chosen sufficiently small later on. Observe that
\begin{align*}
&\PP\left\{X_{C(\sigma)}\in[\eta_k-r;\eta_k+r]\right\}\le\PP\left\{C(\sigma)\geq\exp\Big[\frac{2}{\sigma^2}\left(H_0(\lambda_0)+\delta\right)\Big]\right\}\\
&\quad+\PP\left\{X_{C(\sigma)}\in[\eta_k-r;\eta_k+r],C(\sigma)\le\exp\Big[\frac{2}{\sigma^2}\left(H_0(\lambda_0)+\delta\right)\Big]\right\}=:I_1(\sigma)+I_2(\sigma)\,.
\end{align*}
According to \eqref{proof_victoire2_1},
\begin{equation*}
\lim_{\sigma\rightarrow 0}I_1(\sigma)=\lim_{\sigma\rightarrow 0}\PP\left\{C(\sigma)\geq\exp\Big[\frac{2}{\sigma^2}\left(H_0(\lambda_0)+\delta\right)\Big]\right\}=0\,.
\end{equation*} 
The second term, $I_2(\sigma)$, can be estimated with the upper-bound:
\begin{align*}
&\PP\left\{X_{C(\sigma)}\in[\eta_k-r;\eta_k+r],C(\sigma)\leq\exp\Big[\frac{2}{\sigma^2}\left(H_0(\lambda_0)+\delta\right)\Big]\right\}\\
&\leq\PP\left\{C_{\eta_k-r,\eta_k+r}(\sigma)\leq\exp\Big[\frac{2}{\sigma^2}\left(H_0(\lambda_0)+\delta\right)\Big]\right\}\,.
\end{align*}
Since $\widetilde{H}_0(\eta_k-r,\eta_k+r)>H_0(\lambda_0)$, by taking $\delta$ sufficiently small,
\begin{align*}
&\PP\left\{C_{\eta_k-r,\eta_k+r}(\sigma)\leq\exp\Big[\frac{2}{\sigma^2}\left(H_0(\lambda_0)+\delta\right)\Big]\right\}\\
&\leq\PP\left\{C_{\eta_k-r,\eta_k+r}(\sigma)\leq\exp\Big[\frac{2}{\sigma^2}\left(\widetilde{H}_0(\eta_k-r,\eta_k+r)-\delta\right)\Big]\right\}\,.
\end{align*}
Applying Lemma \ref{prison}, it follows that $\lim_{\sigma\rightarrow 0}I_2(\sigma)=0$, and so:
\begin{equation*} 
\lim_{\sigma\to0}\PP\left\{X_{C(\sigma)}\in[\eta_k-r;\eta_k+r]\right\}=0\,.
\end{equation*}
As such, 
\begin{equation*}
\lim_{\sigma\to0}\PP\left\{X_{C(\sigma)}\in[\lambda_1;\lambda_2]\setminus[\lambda_0-\rho;\lambda_0+\rho]\right\}=0.
\end{equation*}

\noindent{}{\bf Step 3.} We complete the proof with the lower-tail estimate in \eqref{eq:victoire2}. For $\rho>0$, given the event~$\{| X_{C(\sigma)}-\lambda_0|\leq\rho\}$, remark that the following inclusion:
\begin{equation*}
\Big\{C(\sigma)\leq\exp\Big[\frac{2}{\sigma^2}\left(H_0(\lambda_0)-\delta\right)\Big]\Big\}\subset\Big\{C_{\lambda_0+y,\lambda_0+y}(\sigma)\leq\exp\Big[\frac{2}{\sigma^2}\left(H_0(\lambda_0)-\delta\right)\Big]\Big\}\,,
\end{equation*}
holds for any $y$ lying in the interval $(-\rho,\rho)$. Therefore, for any $\rho>0$,
\begin{align*}
&\PP\left\{C(\sigma)\leq\exp\Big[\frac{2}{\sigma^2}\left(H_0(\lambda_0)-\delta\right)\Big]\right\}
\leq\PP\left\{| X_{C(\sigma)}-\lambda_0|\geq\rho\right\}\\
&\quad +\PP\left\{\inf_{y\in[-\rho,\rho]}C_{\lambda_0+y,\lambda_0+y}(\sigma)\leq\exp\Big[\frac{2}{\sigma^2}\left(H_0(\lambda_0)-\delta\right)
\Big]\right\}\,.
\end{align*}
According to {\bf Step 2}, the first upper-bound vanishes as $\sigma\downarrow 0$. The second upper-bound can be estimated by
\begin{align*}
&\PP\left\{\inf_{y\in[-\rho,\rho]}C_{\lambda_0+y,\lambda_0+y}(\sigma)\leq\exp\Big[\frac{2}{\sigma^2}\left(H_0(\lambda_0)-\delta\right)\Big]\right\}\\
&\leq\PP\left\{C_{\lambda_0-\rho,\lambda_0+\rho}(\sigma)\leq\exp\Big[\frac{2}{\sigma^2}\left(H_0(\lambda_0)-\delta\right)\Big]\right\}\,.
\end{align*}
As $\lim_{\rho\rightarrow 0}\Psi_1(\lambda_0-\rho)+\Psi_2(\lambda_0+\rho)= H_0(\lambda_0)$, $\rho$ can be chosen sufficiently small so that~$\Psi_1(\lambda_0-\rho)+\Psi_2(\lambda_0+\rho)\geq H_0(\lambda_0)-\frac{\delta}{2}$. Thus, the limit~\eqref{alberto2} implies
\begin{equation*}
\lim_{\sigma\to0}\PP\left\{C_{\lambda_0-\rho,\lambda_0+\rho}(\sigma)\leq\exp\Big[\frac{2}{\sigma^2}\left(H_0(\lambda_0)-\delta\right)\Big]\right\}=0\,,
\end{equation*}
and so
\begin{equation*}
\lim_{\sigma\to0}\PP\left\{C(\sigma)\leq\exp\Big[\frac{2}{\sigma^2}\left(H_0(\lambda_0)-\delta\right)\Big]\right\}=0\,.
\end{equation*}
\end{proof}
The analog of Theorem \ref{victoire2} for the particle systems $(X^{1,N},Y^{1,N}),\cdots,(X^{N,N},Y^{N,N})$ is given by the following:
 \begin{thm}
\label{victoire3} For $N$ sufficiently large, it holds: 
 for any $\delta>0$, $1\le i\le N$,
\begin{equation}
\label{eq:victoire3}
\lim_{\sigma\to0}\PP\left\{\exp\left[\frac{2}{\sigma^2}\left(H_0(\lambda_0)-\delta\right)\right]<
C^i_N(\sigma)<\exp\left[\frac{2}{\sigma^2}\left(H_0(\lambda_0)+\delta\right)\right]\right\}=1\,,
\end{equation}
and 
\begin{equation*}
\lim_{\sigma\to0}\PP\left\{| X_{C^i_N(\sigma)}^{i,N}-\lambda_0|\leq\delta\right\}=1\,.
\end{equation*}
\end{thm}
\begin{proof} The proof readily follows the main steps of Theorem \ref{victoire2}: owing to Proposition~\ref{lucille2}, Lemma \ref{prison} still holds true for 
\[
C^{i}_{N,z_1,z_2}(\sigma):=\inf\{t\ge 0\,:\,(X^{i,N}_t,Y^{i,N}_t)\notin (\mathbb R\times\mathbb R)\setminus \mathcal D_{z_1}\times \mathcal D_{z_2}\}
\]
 in place of $C_{z_1,z_2}(\sigma)$. From this, since the proof arguments of Theorem \ref{victoire2} essentially rely on the regularity of $H_0$, one can replicate each argument replacing $C(\sigma)$ by $C^{i}_{N}(\sigma)$.
 \end{proof}

 \begin{rem} Although we left aside the linear case considered in Section \ref{sec:LinearCase}, let us point out that, following the same proof arguments, an analog to Theorem \ref{victoire2} can be established for $(x^\sigma,y^\sigma)$: assuming that $\Psi_1$ and $\Psi_2$ {\color{black}are uniformly} convex, of class $\mathcal C^2$ and such that \eqref{Linear:Init} holds, and that, for $\lambda_i={\rm argmin}\Psi_i$, $\lambda_1<\lambda_2$, then, given
$$
c(\sigma)=\inf\left\{t\geq0\,\,:\,\,x_t^\sigma=y_t^\sigma\right\}\,,
$$
and $z_0$ the (unique) minimizer of $h_0(z)=\Psi_1(z)+\Psi_2(z)-\Psi_1(\lambda_1)-\Psi_1(\lambda_2)$,
for any $\delta>0$, we have
\begin{equation*}
\begin{aligned}
&\lim_{\sigma\to0}\PP\left\{\exp\left[\frac{2}{\sigma^2}\left(h_0(z_0)-\delta\right)\right]<c(\sigma)<\exp\left[\frac{2}{\sigma^2}\left(
h_0(z_0)+\delta\right)\right]\right\}=1\,,
\end{aligned}
\end{equation*}
and 
\begin{equation*} 
\lim_{\sigma\to0}\PP\left\{\left|x_{c(\sigma)}^\sigma-z_0\right|\leq\delta\right\}=1\,.
\end{equation*}	
\end{rem}
 
\paragraph{Acknowledgments}
For the first author, the paper was prepared within the framework of the Basic Research Program at HSE University. The second author acknowledges the
support of the French ANR grant
'METANOLIN' (ANR-19-CE40-0009). This work has been initiated at the end of 2020 and finalized during the visit of the first author at the University Jean-Monnet in 2022. The first author expresses his gratitude to all members of the university for their hospitality during this stay.  

\footnotesize{	

}
\end{document}